\documentclass[11pt]{article}
\usepackage{amssymb,amsfonts}
\usepackage{amsmath,amsthm,amsxtra}

\usepackage{graphicx}
\usepackage{amscd}
\usepackage{ulem}
\usepackage{makeidx}
\usepackage{fancybox}

\usepackage[linktocpage=true, colorlinks=true, linkcolor=blue, citecolor=red, urlcolor=green]{hyperref}

\newcommand{\ccc}{{\mathbf C}}

\newcommand{\nnn}{{\mathbf N}}
\newcommand{\qqq}{{\mathbf Q}}
\newcommand{\rrr}{{\mathbf R}}
\newcommand{\zzz}{{\mathbf Z}}

\renewcommand{\ggg}{{\frak{g}}}
\newcommand{\hhh}{{\frak{h}}}

\newtheorem{prop}{Proposition}[section]
\newtheorem{lemma}{Lemma}[section]

\newtheorem{note}{Note}[section]

\numberwithin{equation}{section}

\begin{document}

\title{A note on modular properties of non-principal and
non-coprincipal admissible $C^{(1)}_2$-modules of integer level}

\author{\footnote{12-4 Karato-Rokkoudai, Kita-ku, Kobe 651-1334, 
Japan, \qquad
wakimoto.minoru.314@m.kyushu-u.ac.jp, \hspace{5mm}
wakimoto@r6.dion.ne.jp 
}{ Minoru Wakimoto}}

\date{\empty}

\maketitle

\begin{center}
Abstract
\end{center}
For the affine Lie algebra $C_2^{(1)}$ we study non-principal 
and non-coprincipal admissible modules of integer level and 
their quantum Hamiltonian reduction, and show that they have 
$\Gamma_0(2)$-modular invariance.

\maketitle

\tableofcontents

\section{Introduction}

Admissible representations are introduced in \cite{KW5} as a class of 
irreducible highest weight modules of affine Lie algebras whose characters 
are modular functions, and are classified in \cite{KW1989}. 
They play very important role, in particular, 
in the study of representations of W-algebras associated to affine 
Lie algebras. Among these admissible modules, modular properties of 
principal or coprincipal admissible modules and their 
quantum Hamiltonian reduction have been studied in several literatures 
(e.g., \cite{AE2019}, \cite{KW2024}).
The aim of this paper is to study the characters of other 
admissible modules, namely non-principal and non-coprincipal 
admissible modules, in the simplest case $C^{(1)}_2$. 
It turns out that, in the case treated in this paper, the space of 
these characters are not $SL_2(\zzz)$-invariant but invariant under 
the action of the congruence subgroup $\Gamma_0(2)$.

As is known well by \cite{K1} (integrable case) and \cite{KW2024} 
(principal admissible case), the normalised characters 
$$
{\rm ch}_{L(\Lambda)}(\tau,z.t)
\,\ = \,\ \frac{A_{\Lambda+\rho}(\tau,z,t)}{R(\tau,z,t)}
$$
of irreducible highest weight modules of non-twisted affine Lie algebras 
of given level $K$ span the $SL_2(\zzz)$-invariant space,
where 
$$
A_{\Lambda+\rho} \,\ := \,\ 
q^{\frac{|\overline{\Lambda+\rho}|^2}{2(K+h^{\vee})}}
\sum_{w \in W}\varepsilon(w) e^{w(\Lambda+\rho)}
$$
and $R$ is the denominator 
(for admissible representations the Weyl group $W$ is replaced with 
$W_{\Lambda}$). 
However, for non-principal and non-coprincipal admissible $C_2^{(1)}$-modules 
$L(\Lambda_{n_1,n_2}^{\, [K]})$ considered in this paper, 
Proposition \ref{paper:prop:2025-1105d} shows that,
$$
{\rm ch}_{L(\Lambda_{n_1,n_2}^{\, [K]})}(\tau,z,t)
\, = \,\ \frac{A_{\Lambda_{n_1,n_2}^{\, [K]}+\rho}(\tau,z, \lq \lq 2t")
}{R(\tau,z,t)}
$$
span the $\Gamma_0(2)$-invariant space.

\medskip

In section \ref{sec:QHR:minimal}, we consider the quantum Hamiltonian 
reduction only for the minimal nilpotent element. This is because, 
for admissible $C^{(1)}_2$-modules treated in this paper, 
the quantum Hamiltonian reduction associated to a non-zero 
nilpotent element other than the minimal nilpotent element vanishes. 
This phenominon seems to be in good matching with 
Corollary 1 in \cite{KW2024} in the principal admissible case 
or with Corollary 6.1 in \cite{KW2025} in the coprincipal 
(=subprincipal) admissible case, 
since in this paper we consider the case $u=1$, which is the smallest 
possible value for positive odd integers $u$. 
It may be expected that the quantum reduction associated 
to other nilpotent elements will take place 
with $\Gamma_0(2)$-modular invariance just like 
Corollary 1 in \cite{KW2024} and Corollary 6.1 in \cite{KW2025}
if we consider admissible modules for general $u \in \nnn_{\rm odd}$.

\section{Non-principal and non-coprincipal admissible 
$C^{(1)}_2$-modules of integer level}
\label{sec:C2:character:admissible}

In this paper, we consider the affine Lie algebra $\ggg = \ggg(C_2^{(1)})$ 
of type $C_2^{(1)}$ with simple root system \, $\Pi =$
\hspace{-5mm}
\setlength{\unitlength}{1mm}
\begin{picture}(36,9)
\put(6,1){\circle{3}}
\put(17,1){\circle{3}}
\put(28,1){\circle{3}}
\put(7.6,1.6){\vector(1,0){8}}
\put(7.6,0.4){\vector(1,0){8}}
\put(26.5,1.6){\vector(-1,0){8}}
\put(26.5,0.4){\vector(-1,0){8}}
\put(6,5){\makebox(0,0){$\alpha_0$}}
\put(17,5){\makebox(0,0){$\alpha_1$}}
\put(28,5){\makebox(0,0){$\alpha_2$}}
\put(11,-2){\makebox(0,0){$-1$}}
\put(22,-2){\makebox(0,0){$-1$}}
\end{picture}
and inner products 
$$
\big((\alpha_i|\alpha_j)\big)_{i,j=0,1,2}
= \begin{pmatrix}
2 & -1 & 0 \\
-1 & 1 & -1 \\
0 & -1 & 2
\end{pmatrix}
$$
For notations and basic properties of affine Lie algebras, 
we follow from the Kac's book \cite{K1}.
The set $\Delta^{re}$ of all real (=non-isotropic) roots of 
$C_2^{(1)}$ is
$$
\Delta^{re} = \,\ \{
n\delta+\alpha \,\ ; \,\ \alpha \in \overline{\Delta} \,\ 
{\rm and} \,\ n \in \zzz\}
$$
where
$$
\overline{\Delta} \,\ = \,\ 
\underbrace{\{\alpha_1, \,\ \alpha_2, \,\ 
\alpha+\alpha_2, \,\ 2\alpha_1+\alpha_2\}}_{
{\displaystyle \overline{\Delta}_+
}} \,\ \cup \,\ 
(-\overline{\Delta}_+)
$$
and $\delta := \alpha_0+2\alpha_1+\alpha_2$
is the primitive imaginary root.

\medskip

For a real root $\alpha$, let 
$\alpha^{\vee}:= \frac{2}{|\alpha|^2}\alpha$ denote its coroot.
Then we have 
$$
\left\{
\begin{array}{lcr}
\alpha_1^{\vee} &=& 2 \, \alpha_1 \\[0.5mm]
\alpha_2^{\vee} &=& \alpha_2
\end{array}\right. \hspace{5mm}
\left\{
\begin{array}{rcccl}
(\alpha_1+\alpha_2)^{\vee} &=& 2 \, (\alpha_1+\alpha_2) 
&=& \alpha_1^{\vee}+ 2\alpha_2^{\vee} 
\\[0.5mm]
(2\alpha_1+\alpha_2)^{\vee} &=& 2 \alpha_1+\alpha_2
&=& \alpha_1^{\vee}+ \alpha_2^{\vee}
\end{array}\right.
$$
and the simple coroot system $\Pi^{\vee} \, : $ 
\hspace{-5mm}
\setlength{\unitlength}{1mm}
\begin{picture}(36,9)
\put(6,1){\circle{3}}
\put(17,1){\circle{3}}
\put(28,1){\circle{3}}
\put(15.4,1.6){\vector(-1,0){8}}
\put(15.4,0.4){\vector(-1,0){8}}
\put(18.6,1.6){\vector(1,0){8}}
\put(18.6,0.4){\vector(1,0){8}}
\put(6,5){\makebox(0,0){$\alpha_0^{\vee}$}}
\put(17,5){\makebox(0,0){$\alpha_1^{\vee}$}}
\put(28,5){\makebox(0,0){$\alpha_2^{\vee}$}}
\put(11,-2){\makebox(0,0){$-2$}}
\put(22,-2){\makebox(0,0){$-2$}}
\end{picture}
and inner products 
$$
\big((\alpha_i^{\vee}|\alpha_j^{\vee})\big)_{i,j=0,1,2}
= \begin{pmatrix}
2 & -2 & 0 \\
-2 & 4 & -2 \\
0 & -2 & 2
\end{pmatrix}
$$
The set of all real coroots of $C^{(1)}_2$ is 
{\allowdisplaybreaks
\begin{eqnarray*}
\Delta^{\vee re} &=& 
\big\{ \alpha^{\vee} \,\ ; \,\ \alpha \in \Delta^{re}\big\}
\\[1mm]
&=& \Big\{
2nc \pm \alpha_1^{\vee}, \,\ 
2nc \pm (\alpha_1^{\vee}+2\alpha_2^{\vee}), \,\ 
nc \pm \alpha_2^{\vee}, \,\ 
nc \pm (\alpha_1^{\vee}+\alpha_2^{\vee}) \,\ ; \,\ 
n \, \in \, \zzz\Big\}
\end{eqnarray*}}
where 
$c := 
\alpha_0^{\vee}+\alpha_1^{\vee}+\alpha_2^{\vee}
\, = \, 
\alpha_0+2\alpha_1+\alpha_2 \, = \, \delta.$

\medskip

The Cartan subalgebra $\hhh$ of $\ggg$, being identified with 
its dual space $\hhh^{\ast}$ via the inner product 
$( \cdot | \cdot)$, is 
\begin{equation}
\hhh \,\ = \,\ \ccc \, \Lambda_0
\, \oplus \, \overline{\hhh}
\, \oplus \, \ccc \, \delta
\label{paper:eqn:2025-1031b}
\end{equation}
where $\overline{\hhh}=\ccc \alpha_1 \oplus \ccc \alpha_2$
and $\Lambda_0$ is the element satisfying 
$(\Lambda_0|\alpha_j)=\delta_{j,0}$ and 
$(\Lambda_0|\Lambda_0)=0$. Let $\Lambda_i$ be the elements 
satisfying $(\Lambda_i|\alpha_j^{\vee})=\delta_{i,j}$ 
and $(\Lambda_i|\Lambda_0)=0$ for all $i,j \in \{0,1,2\}$. 
For $\lambda \in \hhh$, let $\overline{\lambda}$ denote 
its $\overline{\hhh}$-component with respect to the
decomposition \eqref{paper:eqn:2025-1031b}.
Note that $\Lambda_i =\Lambda_0+\overline{\Lambda}_i$ for 
$i=1,2$.
In the case $C_2^{(1)}$, the Weyl vector
$\rho =\sum_{i=0}^2\Lambda_i$ and the dual Coxeter number 
$h^{\vee} = (\rho|\delta)$ are as follows:
\begin{equation}
\rho=3\Lambda_0+2\alpha_1+\frac32 \alpha_2, \quad 
|\rho|^2=\frac52 \quad {\rm and} \quad 
h^{\vee}=3
\label{paper:eqn:2025-1031c}
\end{equation}
Let $\Delta_+^{re}$ (resp. $\Delta^{\vee re}_+$) be the set of 
positive real roots (resp. coroots), namely,
$$
\Delta_+^{re} = \{\alpha \in \Delta^{re} \,\ ; \,\ (\rho|\alpha)>0\} 
\quad {\rm and} \quad 
\Delta_+^{\vee re} = \{\alpha^{\vee} \in \Delta^{\vee re} \,\ ; \,\ 
(\rho|\alpha^{\vee})>0\}
$$

\subsection{Non-principal and non-coprincipal admissible weights}
\label{subsec:C2:admissible:weights}

For a positive integer $u \in \nnn$, define subsets 
$\Delta^{\vee re}_{(u)}$ and $S_{(u)}^{\vee}$ of
$\Delta^{\vee re}$ by 
$$
\begin{array}{lcl}
\Delta^{\vee re}_{(u)} &:=& \big\{
2nuc \pm \alpha_1^{\vee}, \,\ 
2nuc \pm (\alpha_1^{\vee}+2\alpha_2^{\vee}), \,\ 
\,\ ; \,\ 
n \, \in \, \zzz\big\}
\\[2.5mm]
S_{(u)}^{\vee} &:=& \big\{
2uc - \alpha_1^{\vee}, \,\ \alpha_1^{\vee}\big\}
\, \cup \, 
\big\{2uc - (\alpha_1^{\vee}+2\alpha_2^{\vee}), \,\ 
\alpha_1^{\vee}+2\alpha_2^{\vee}\big\}
\end{array}
$$
We note that $S_{(u)}^{\vee}$ is the simple coroot system of 
$A^{(1)}_1 \oplus A^{(1)}_1$.
The Weyl group $W_{(u)}$ of $\Delta_{(u)}^{\vee re}$ is
$$
W_{(u)} \, = \, 
\big\langle r_{\alpha_1}, \, r_{\alpha_1+\alpha_2}\rangle 
\ltimes
\big\{t_{2ju\alpha_1+2ku(\alpha_1+\alpha_2)} \,\ ; \,\ 
j, \, k \, \in \, \zzz \big\}
$$
where $t_{\alpha}$ is the transformation on $\hhh$ defined in
\cite{K1}:
$$
t_{\alpha}(\lambda)=
\lambda+(\lambda|\delta)\alpha
-
\Big\{\frac{|\alpha|^2}{2}(\lambda|\delta)
+
(\lambda|\alpha)\Big\} \, \delta
$$

For an element $\Lambda \in \hhh^{\ast}$, we put 
\begin{equation}
\Delta^{\vee re}_{\Lambda} \, := \,\ 
\{\alpha^{\vee} \,\ ; \,\ 
(\Lambda+\rho| \, \alpha^{\vee}) \, \in \, \zzz\}
\label{paper:eqn:2025-1031d1}
\end{equation}
and
\begin{equation}
\Pi^{\vee}_{\Lambda} \, := \,\ 
\text{the set of simple elements in} \, 
\Delta^{\vee re}_{\Lambda} \cap \Delta^{\vee re}_+
\label{paper:eqn:2025-1031d2}
\end{equation}

\medskip


\begin{lemma} \,\ 
\label{paper:lemma:2025-1031a}
Let $\Lambda \in \hhh^{\ast}$, $u \in \nnn$ and $K \in \qqq$
be elements satisfying the conditions
\begin{subequations}
{\allowdisplaybreaks
\begin{eqnarray}
(\Lambda+\rho| \, \alpha^{\vee}) & \not\in & \zzz_{\leq 0}
\hspace{10mm} 
{\rm for} \,\ {}^{\forall} \alpha^{\vee} \in 
\Delta^{\vee re}_+
\label{paper:eqn:2025-1101a1}
\\[1mm]
\Delta_{\Lambda}^{\vee re} &=& \Delta^{\vee re}_{(u)}
\label{paper:eqn:2025-1101a2}
\\[1mm]
(\Lambda| \, \delta) &=& K
\label{paper:eqn:2025-1101a3}
\end{eqnarray}}
\end{subequations}
Then
\begin{enumerate}
\item[{\rm 1)}] \,\ $u \, \in \, \nnn_{\rm odd}$

\item[{\rm 2)}] \,\ $K+h^{\vee} \, = \, \dfrac{p}{u}$ \quad
{\rm where} \quad $p \in \nnn_{\geq 2}$ \,\ {\rm and} \,\ 
$\gcd(u,p)=1$

\item[{\rm 3)}] \,\ $\Lambda+\rho \,\ = \,\ 
(K+h^{\vee})\Lambda_0
\, + \, 
\dfrac{n_1}{2} \, \alpha_1 
\, + \, 
\dfrac{n_2}{2} \, (\alpha_1+\alpha_2)$

\noindent
where $n_1, \, n_2 \, \in \, \nnn$ \,\ such that \,\ 
$n_1, \, n_2 \, < \, 2p$ \,\ and \,\ 
$n_1 \not\equiv n_2 \mod 2$
\end{enumerate}
\end{lemma}

\begin{proof} \,\ By conditions \eqref{paper:eqn:2025-1101a1} 
and \eqref{paper:eqn:2025-1101a2}, we have 
\begin{subequations}
{\allowdisplaybreaks
\begin{eqnarray}
(\Lambda+\rho \, | \, \alpha^{\vee}) &\not\in& - \, \zzz_{\geqq 0}
\hspace{10mm} ({}^{\forall} \alpha^{\vee} \, \in \, \Delta^{\vee}_+)
\label{paper:eqn:2025-1101c1}
\\[2mm]
(\Lambda+\rho \, | \, 2uc-\alpha_1^{\vee})
&=& 2u(K+h^{\vee})
-(\Lambda+\rho|\alpha_1^{\vee}) \, \in \, \nnn
\label{paper:eqn:2025-1101c2}
\\[2mm]
(\Lambda+\rho \, | \, 2uc-(\alpha_1^{\vee}+2\alpha_2^{\vee}))
&=& 2u(K+h^{\vee})
-(\Lambda+\rho|\alpha_1^{\vee}+2\alpha_2^{\vee}) \, \in \, \nnn
\label{paper:eqn:2025-1101c3}
\\[2mm]
(\Lambda+\rho \, | \, nc \pm \alpha_2^{\vee})
&=& n(K+h^{\vee})
\pm (\Lambda+\rho|\alpha_2^{\vee}) \, \not\in \, \zzz
\hspace{10mm} ({}^{\forall}n \, \in \, \zzz)
\label{paper:eqn:2025-1101c4}
\\[2mm]
(\Lambda+\rho \, | \, nc \pm (\alpha_1^{\vee}+\alpha_2^{\vee}))
&=& n(K+h^{\vee})
\pm (\Lambda+\rho|\alpha_1^{\vee}+\alpha_2^{\vee}) \, \not\in \, \zzz
\hspace{5mm} ({}^{\forall}n \, \in \, \zzz)
\label{paper:eqn:2025-1101c5}
\end{eqnarray}}
\end{subequations}
Putting \, $p' := 2u(K+h^{\vee})$, these equations imply that
\begin{subequations}
\begin{equation}
\left\{
\begin{array}{lcc}
(\Lambda+\rho|\alpha_1^{\vee}) &\in & \nnn \\[2mm]
(\Lambda+\rho|\alpha_1^{\vee}+2\alpha_2^{\vee}) &\in & \nnn \\[2mm]
(\Lambda+\rho|\alpha_2^{\vee}) &\in & \frac12 \, \zzz_{\rm odd}
\end{array}\right. \hspace{10mm} \left\{
\begin{array}{ccccc}
p' &\in & \nnn & & \\[1mm]
K+h^{\vee} &=& \dfrac{p'}{2u} & & \\[3mm]
n \cdot \dfrac{p'}{2u} &\not\in & \frac12 \, \zzz_{\rm odd}
& & ({}^{\forall}n \, \in \, \zzz)
\end{array}\right. 
\label{paper:eqn:2025-1101b1}
\end{equation}
From \eqref{paper:eqn:2025-1101b1}, we have 
\begin{equation}
\left\{
\begin{array}{lclcc}
p' &\in & \nnn_{\rm even} & {\rm i.e,} & p \, := \, \frac{p'}{2} \, \in \, \nnn
\\[1mm]
u &\in & \nnn_{\rm odd} & &
\end{array}\right.  \qquad {\rm then} \,\ K+h^{\vee}=\dfrac{p}{u}
\label{paper:eqn:2025-1101b2}
\end{equation}
Also, by \eqref{paper:eqn:2025-1101c2} and \eqref{paper:eqn:2025-1101c3},
we have 
\begin{equation}
\left\{
\begin{array}{lcc}
p'-(\Lambda+\rho \, | \, \alpha_1^{\vee}) &\in & \nnn
\\[2mm]
p'-(\Lambda+\rho \, | \, \alpha_1^{\vee}+2\alpha_2^{\vee}) &\in & \nnn
\end{array}\right.
\label{paper:eqn:2025-1101b3}
\end{equation}
\end{subequations}

In order to prove ${\rm gcd}(u,p')=1$, assume that 
$u=au_0$ and $p'=ap'_0$ where $a \in \nnn_{\geq 2}$ and 
$u_0, p'_0 \in \nnn$. Then, by $2u(K+h^{\vee})=p'$, 
we have $2u_0(K+h^{\vee})=p'_0$ \, so
$$
(\Lambda+\rho| \, 2u_0c-\alpha^{\vee}_1) \, = \, 
2u_0(K+h^{\vee})-(\Lambda+\rho| \, \alpha^{\vee}_1) 
\,\ \in \,\ \zzz
$$
This means that 
$$
2u_0c-\alpha^{\vee}_1 \, \in \, \Delta_{\Lambda}^{\vee re}
\qquad {\rm so} \qquad 
2u_0c-\alpha^{\vee}_1 \, \in \, \Delta_{(u)}^{\vee re} \, ,
$$
but this is impossible since $0<u_0<u$.

\medskip

In order to prove $p' \geq 4$ and the claim 3), we put 
\begin{equation}
(\Lambda+\rho| \, \alpha_1^{\vee}) \, =: \, n_1 \qquad 
{\rm and} \qquad 
(\Lambda+\rho| \, \alpha_1^{\vee}+2\alpha_1^{\vee}) \, =: \, n_2
\label{paper:eqn:2025-1101d}
\end{equation}
Then, by \eqref{paper:eqn:2025-1101b1} and 
\eqref{paper:eqn:2025-1101b3}, we have 
\begin{equation}
\left\{
\begin{array}{ccc}
n_1, \,\ n_2 &\in & \nnn \\[0mm]
n_1-n_2 & \in & \zzz_{\rm odd}
\end{array}\right. \hspace{5mm} {\rm and} \hspace{5mm}
1 \, \leq \, n_1, \,n_2 \, \leq \, p'-1
\label{paper:eqn:2025-1101e1}
\end{equation}
From this, we have 
\begin{equation}
2 \, \leq \, 
\underbrace{n_1+n_2}_{\substack{ \rotatebox{-90}{$\geq$}
\\[0mm] {\displaystyle 3
}}} \, \leq \, 2(p'-1) \hspace{5mm} 
\therefore \quad 2p' \, \geq \, 5 \hspace{5mm} 
\underset{\substack{\\[0.5mm] \uparrow \\[1mm] 
p' \, \text{is even}
}}{\therefore} p' \, \geq \, 4
\label{paper:eqn:2025-1101e2}
\end{equation}

Since 
$$
\begin{pmatrix}
(\alpha_1| \, \alpha_1^{\vee}) & 
(\alpha_1| \, \alpha_1^{\vee}+2\alpha_2^{\vee})
\\[1mm]
(\alpha_1+\alpha_2| \, \alpha_1^{\vee}) & 
(\alpha_1+\alpha_2| \, \alpha_1^{\vee}+2\alpha_2^{\vee})
\end{pmatrix}
= 2 \, 
\begin{pmatrix}
(\alpha_1| \, \alpha_1) & (\alpha_1| \, \alpha_1+\alpha_2)
\\[1mm]
(\alpha_1+\alpha_2| \, \alpha_1) & 
(\alpha_1+\alpha_2| \, \alpha_1+\alpha_2)
\end{pmatrix}
= 
\begin{pmatrix}
2 & 0 \\[1mm]
0 & 2
\end{pmatrix} ,
$$
the equation \eqref{paper:eqn:2025-1101d} means that
$$
\overline{\Lambda+\rho} \,\ = \,\ 
\frac{n_1}{2} \, \alpha_1 \, + \, 
\frac{n_2}{2} \, (\alpha_1+\alpha_2)
$$
so 
$$
\Lambda+\rho \,\ = \,\ (K+h^{\vee})\Lambda_0
\, + \, 
\frac{n_1}{2} \, \alpha_1 
\, + \, 
\frac{n_2}{2} \, (\alpha_1+\alpha_2)
$$
proving 3). 
\end{proof}

\medskip 
An element $\Lambda \in \hhh^{\ast}$ satisfying the conditions 
\eqref{paper:eqn:2025-1101a1} $\sim$ 
\eqref{paper:eqn:2025-1101a1} is called \lq \lq an admissible 
weight of type $S^{\vee}_{(u)}$ and of level $K$".
For $K \in \qqq$ and $n_1, \, n_2 \, \in \, \nnn$, we put

\begin{subequations}
\begin{equation}
\Lambda^{[K]}_{n_1,n_2} \, := \,\ 
K\Lambda_0 
\, + \, \frac{n_1-1}{2} \, \alpha_1
\, + \, \frac{n_2-3}{2} \, (\alpha_1+\alpha_2) 
\label{paper:eqn:2025-1101f1}
\end{equation}
so that 
\begin{equation}
\Lambda^{[K]}_{n_1,n_2}+\rho \, = \,\ 
(K+h^{\vee})\Lambda_0 
\, + \, \frac{n_1}{2} \, \alpha_1
\, + \, \frac{n_2}{2} \, (\alpha_1+\alpha_2) 
\label{paper:eqn:2025-1101f2}
\end{equation}
\end{subequations}
Then, by Lemma \ref{paper:lemma:2025-1031a}, we obtain the 
following:

\medskip

\begin{prop} \,\
\label{paper:prop:2025-1031a}
For $K \in \qqq$ and $u \in \nnn$, we put 
\begin{equation}
P_K^+(S^{\vee}_{(u)}) := 
\text{the set of admissible weights
$\Lambda$ of type $S^{\vee}_{(u)}$ and of level $K$}.
\label{paper:eqn:2025-1101g}
\end{equation}
Then
\begin{enumerate}
\item[{\rm 1)}] \,\ $P_K^+(S^{\vee}_{(u)}) 
\, \ne \, \emptyset
\quad \Longleftrightarrow \quad 
u \in \nnn_{\rm odd}, \,\ u(K+3) \in \nnn_{\geq 2} \,\ 
\text{and} \,\ \gcd(u, u(K+3))=1$

\item[{\rm 2)}] \,\ In the case when $K$ and $u$ satisfy the
conditions in the above {\rm 1)},
$$
P_K^+(S^{\vee}_{(u)}) \,\ = \,\ 
\left\{\Lambda^{[K]}_{n_1,n_2} 
\,\ ; \,\ 
\begin{array}{l}
n_1, \, n_2 \, \in \, \nnn \\[1mm]
n_1, \, n_2 \, < \, 2u(K+3) \\[1mm]
n_1-n_2 \, \in \, \zzz_{\rm odd}
\end{array} \right\} 
$$
\item[{\rm 3)}] \,\ If $\Lambda^{[K]}_{n_1,n_2} 
\, \in \, P^+_K(S^{\vee}_{(u)})$, then 
$$
\Lambda^{[K]}_{2u(K+3)-n_1, \, n_2}, \,\ 
\Lambda^{[K]}_{n_1, \, 2u(K+3)-n_2}, \,\ 
\Lambda^{[K]}_{2u(K+3)-n_1, \, 2(K+3)-n_2} \, \in \, 
P^+_K(S^{\vee}_{(u)})
$$
\end{enumerate}
\end{prop}

\medskip

We note that, in particular when $u=1$, 
\begin{equation}
P_K^+(S^{\vee}_{(u=1)}) 
\, \ne \, \emptyset
\quad \Longleftrightarrow \quad 
K \in \zzz_{\geq -1}
\label{paper:eqn:2025-1101h}
\end{equation}
Hereandafter in this paper, we consider the case $u=1$, 
namely admissible weights of type $S^{\vee}_{(u=1)}$ 
and compute the character of these modules and 
quantum Hamiltonian reduction.

\subsection{Calculation of numerators of characters of admissible 
modules in the case $u=1$}
\label{subsec:C2:another:u=1:numerator}


\begin{lemma} \,\ 
\label{paper:lemma:2025-1031b}
For $\Lambda=\Lambda^{[K]}_{n_1,n_2} \in P_K^+(S^{\vee}_{(u=1)})$, 
we define the functions $F^{[K]}_{n_1,n_2}$ and 
$A^{[K]}_{n_1,n_2}$ by
{\allowdisplaybreaks
\begin{eqnarray*}
F^{[K]}_{n_1,n_2} &:=& F_{\Lambda+\rho} 
\, := \,\ 
q^{\frac{|\overline{\Lambda+\rho}|^2}{2(K+3)}} 
\sum_{j, \, k \, \in \, \zzz}
t_{2j\alpha_1+2k(\alpha_1+\alpha_2)}(e^{\Lambda+\rho})
\\[1mm]
A^{[K]}_{n_1,n_2} &:=& 
A_{\Lambda+\rho} \, := \,\ 
q^{\frac{|\overline{\Lambda+\rho}|^2}{2(K+3)}} 
\sum_{w \in W_{(u=1)}} \hspace{-3mm}
\varepsilon(w) \, e^{w(\Lambda+\rho)}
\, = 
\sum_{\overline{w} \, \in \, 
\langle r_{\alpha_1}, \, r_{\alpha_1+\alpha_2}\rangle}
\hspace{-7mm}
\varepsilon(\overline{w}) \, 
\overline{w}(F_{\Lambda+\rho})
\end{eqnarray*}}
Then
\begin{subequations}
\begin{enumerate}
\item[{\rm 1)}] \,\ $F^{[K]}_{n_1,n_2} \, = \,\ 
e^{(K+3)\Lambda_0}
\sum\limits_{j, \, k \, \in \, \zzz}
e^{2(K+3)(j+\frac{n_1}{4(K+3)}) \alpha_1
+ 2(K+3)(k+\frac{n_2}{4(K+3)}) (\alpha_1+\alpha_2)}$
\begin{equation}
\times \,\ 
q^{2(K+3)(j+\frac{n_1}{4(K+3)})^2+2(K+3)(k+\frac{n_2}{4(K+3)})^2}
\label{paper:eqn:2025-1101k1}
\end{equation}

\item[{\rm 2)}] \,\ $A^{[K]}_{n_1,n_2} = \, 
e^{(K+3)\Lambda_0}
\sum\limits_{j, \, k \, \in \, \zzz}
\big(e^{2(K+3)(j+\frac{n_1}{4(K+3)}) \alpha_1}
- 
e^{-2(K+3)(j+\frac{n_1}{4(K+3)}) \alpha_1}\big)$
{\allowdisplaybreaks
\begin{eqnarray}
&\times & 
\big(e^{2(K+3)(k+\frac{n_2}{4(K+3)}) (\alpha_1+\alpha_2)} 
-
e^{-2(K+3)(k+\frac{n_2}{4(K+3)}) (\alpha_1+\alpha_2)}\big)
\nonumber
\\[2mm]
&\times &
q^{2(K+3)(j+\frac{n_1}{4(K+3)})^2+2(m+3)(k+\frac{n_2}{4(K+3)})^2}
\label{paper:eqn:2025-1101k2}
\end{eqnarray}}
\end{enumerate}
\end{subequations}
\end{lemma}

\begin{proof} 1) is obtained by easy calculation noticing
$\big|\overline{\Lambda^{[K]}_{n_1,n_2}+\rho}\big|^2 \, = \,\ 
\frac14(n_1^2+n_2^2)$ and using
$$
\begin{array}{lcl}
t_{2j\alpha_1+2k(\alpha_1+\alpha_2)}(\Lambda_0) &=& 
\Lambda_0
\, + \, 2j\alpha_1
\, + \, 2k(\alpha_1+\alpha_2)
\, - \, 2(j^2+k^2) \, \delta
\\[2mm]
t_{2j\alpha_1+2k(\alpha_1+\alpha_2)}(\alpha_1) &=&
\alpha_1-2j \, \delta
\\[2mm]
t_{2j\alpha_1+2k(\alpha_1+\alpha_2)}(\alpha_1+\alpha_2)
&=&
\alpha_1+\alpha_2-2k \, \delta \, .
\end{array}
$$
2) follows immediately from 1).
\end{proof}

\medskip

For our further calculation, we introduce the coordinates in $\hhh$ as follows:
\begin{subequations}
\begin{equation}
h \, = \, (\tau, z_1,z_2,t) \,\ = \,\ 2\pi i \, \big\{
-\tau \, \Lambda_0 \, + \, 
z_1 \, \alpha_1 \, + \, z_2 \, (\alpha_1+\alpha_2)
\, + \, t \, \delta\big\}
\label{paper:eqn:2025-1101h1}
\end{equation}
then,
\begin{equation} \hspace{-10mm}
\left\{
\begin{array}{lcl}
e^{\alpha_1(h)} &=& e^{2\pi iz_1} \\[0.5mm]
e^{(\alpha_1+\alpha_2)(h)} &=& e^{2\pi iz_2}
\end{array}\right. \hspace{10mm}
\left\{
\begin{array}{lcl}
e^{\alpha_2(h)} &=& e^{2\pi i(-z_1+z_2)} \\[0.5mm]
e^{(2\alpha_1+\alpha_2)(h)} &=& e^{2\pi i(z_1+z_2)}
\end{array}\right.
\label{paper:eqn:2025-1101h2}
\end{equation}
\end{subequations}
Then the formula \eqref{paper:eqn:2025-1101k2} is written 
in terms of these coordinates and by using the Jacobi's theta 
functions as follows:
\begin{equation}
A_{n_1,n_2}^{\, [K]}(\tau, z_1, z_2,t)
= 
e^{2\pi i(K+3)t} 
\big[\theta_{n_1,2(K+3)}-\theta_{-n_1,2(K+3)}\big](\tau,z_1) \cdot
\big[\theta_{n_2,2(K+3)}-\theta_{-n_2,2(K+3)}\big](\tau,z_2)
\label{paper:eqn:2025-1031j}
\end{equation}

\section{The congruence subgroup $\Gamma_0(2)$ of $SL_2(\zzz)$}
\label{sec:sl(2Z):congruence-subgroup}

\subsection{Action of $\Gamma_0(2)$ on theta functions}
\label{subsec:sl(2Z):congruence-subgroup}

Let us recall the congruence subgroup $\Gamma_0(2)$ of 
$SL_2(\zzz)$ :
$$
\Gamma_0(2) \, := \,\ \Bigg\{
\begin{pmatrix}
a & b \\[0mm]
c & d
\end{pmatrix} \, \in \, SL_2(\zzz) 
\quad ; \quad c \, \equiv \, 0 \,\ {\rm mod} \, 2 \Bigg\}
$$
It is known that the group $\Gamma_0(2)$ is 
generated by 
$T= \begin{pmatrix}
1 & 1 \\[0mm]
0 & 1
\end{pmatrix}$, 
$ST^2S = \begin{pmatrix}
-1 & 0 \\[0mm]
2 & -1
\end{pmatrix}$ and $-I$, where $S= \begin{pmatrix}
0 & -1 \\[0mm]
1 & 0
\end{pmatrix}$.

\medskip

Following the Kac's book \cite{K1} \S 13.4, define the action 
of $SL_2(\zzz)$ on functions $F(\tau,z,t)$, where
$(\tau,z,t) \in \ccc_+ \times V \times \ccc$ and 
$V$ is an $\ell$-dimensional $\ccc$-linear space with 
inner product $( \cdot | \cdot )$, by
\begin{equation}
F|_A(\tau,z,t) := 
(c\tau+d)^{-\frac{\ell}{2}} 
F\Big(\frac{a\tau+b}{c\tau+d}, \, \frac{z}{c\tau+d}, \, 
t-\frac{c(z|z)}{2(c\tau+d)}\Big)
\,\ {\rm where} \, 
A=
\begin{pmatrix}
a & b \\[0mm]
c & d
\end{pmatrix} \in SL_2(\zzz)
\label{paper:eqn:2025-1031a}
\end{equation}
For convenience of our calculation using this formula 
\eqref{paper:eqn:2025-1031a}, 
we define the functions $\widetilde{\eta}(\tau,t)$ and 
$\widetilde{\theta}_{j,m}(\tau,z,t)$  and 
$\widetilde{\vartheta}_{ab}(\tau,z,t)$ by adding the 
parameter \lq \lq $t$" to Dedekind's $\eta$-function and 
Jacobi's theta functions and Mumford's theta functions as follows:
\begin{equation}
\begin{array}{lcl}
\widetilde{\eta}(\tau,t) &:=& e^{\pi it} \, \eta(\tau)
\\[1mm]
\widetilde{\theta}_{j,m}(\tau,z,t) &:=& 
e^{\pi imt} \, \theta_{j,m}(\tau,z)
\,\ = \,\ 
e^{\pi imt} \, \sum\limits_{n \in \zzz}
e^{2\pi im(n+\frac{j}{2m})z}
q^{m(n+\frac{j}{2m})^2}
\\[1mm]
\widetilde{\vartheta}_{ab}(\tau,z,t) &:=& 
e^{2\pi it} \, \vartheta_{ab}(\tau,z)
\\[2mm]
& & \hspace{-25mm}
= \, 
e^{2\pi it} \, 
q^{\frac{a}{8}}
e^{-a\pi i(z+\frac{b}{2})}
\prod\limits_{n=1}^{\infty}
(1-q^n)
\big(1+(-1)^be^{2\pi iz}q^{n-\frac{a+1}{2}}\big)
\big(1+(-1)^be^{-2\pi iz}q^{n+\frac{a-1}{2}}\big)
\end{array}
\label{paper:eqn:2025-1102a}
\end{equation}
Modular transformation properties of these functions are 
obtained easily from those of $\eta(\tau)$ and 
$\theta_{j,m}(\tau,z)$ and $\vartheta_{ab}(\tau,z)$ as follows:

\medskip

\begin{note} \,\ 
\label{paper:note:2025-1101a}
\begin{enumerate}
\item[{\rm 1)}] \,\ $\widetilde{\eta}|_S(\tau,t) \, = \, 
e^{-\frac{\pi i}{4}} \, \widetilde{\eta}(\tau,t)$ 
\quad and \quad 
$\widetilde{\eta}|_T(\tau,t) \, = \, 
e^{\frac{\pi i}{12}} \, \widetilde{\eta}(\tau,t)$ 

\item[{\rm 2)}] For $m \in \nnn$ and $j \in \zzz$, the following
formulas hold:
\begin{enumerate}
\item[{\rm (i)}] \,\ $\widetilde{\theta}_{j,m}|_S(\tau,z,t) 
\, = \,\ 
\dfrac{e^{-\frac{\pi i}{4}}}{\sqrt{2m}} 
\sum\limits_{k \, \in \, \zzz/2m\zzz}
e^{-\frac{\pi ijk}{m}} \, \widetilde{\theta}_{k,m}(\tau,z,t)$

\item[{\rm (ii)}] \,\ $\widetilde{\theta}_{j,m}|_T(\tau,z,t) 
\,\ = \,\ e^{\frac{\pi ij^2}{2m}} \, \widetilde{\theta}_{j,m}(\tau,z,t)$
\end{enumerate}

\item[{\rm 3)}] For $a, \, b \in \{0,1\}$, the following
formulas hold:
\begin{enumerate}
\item[{\rm (i)}] \quad $\widetilde{\vartheta}_{ab}|_S(\tau,z,t) 
\,\ = \,\
e^{-\frac{\pi i}{4}} \, (-i)^{ab} \, \widetilde{\vartheta}_{ba}(\tau,z,t)$

\item[{\rm (ii)}] \quad $\widetilde{\vartheta}_{ab}|_T(\tau,z,t) 
\,\ = \,\
\left\{
\begin{array}{rcl}
\widetilde{\vartheta}_{a, 1-b}(\tau,z,t) & & {\rm if} \,\ a=0 \\[1mm]
e^{\frac{\pi i}{4}} \, \widetilde{\vartheta}_{ab}(\tau,z,t) 
& & {\rm if} \,\ a=1
\end{array}\right. $

{\rm i.e,}
$$
\left\{
\begin{array}{ccc}
\widetilde{\vartheta}_{00}|_T &=& \widetilde{\vartheta}_{01} \\[1mm]
\widetilde{\vartheta}_{01}|_T &=& \widetilde{\vartheta}_{00}
\end{array}\right. \hspace{10mm} \left\{
\begin{array}{ccc}
\widetilde{\vartheta}_{10}|_T &=& 
e^{\frac{\pi i}{4}} \, \widetilde{\vartheta}_{10} \\[1mm]
\widetilde{\vartheta}_{11}|_T &=& 
e^{\frac{\pi i}{4}} \, \widetilde{\vartheta}_{11}
\end{array}\right.
$$

\item[{\rm (iii)}] \quad $\widetilde{\vartheta}_{ab}|_{T^2}(\tau,z,t) 
\,\ = \,\
i^a \, \widetilde{\vartheta}_{ab}(\tau,z,t)
\,\ = \,\ 
\left\{
\begin{array}{rcl}
\widetilde{\vartheta}_{ab}(\tau,z,t) & & {\rm if} \,\ a=0 \\[1mm]
i \, \widetilde{\vartheta}_{ab}(\tau,z,t) & & {\rm if} \,\ a=1
\end{array}\right. $
\end{enumerate}
\end{enumerate}
\end{note}

\medskip

From Note \ref{paper:note:2025-1101a}, the $ST^2S$-transformation 
of these functions is easily obtained as follows:

\medskip

\begin{note} \,\ 
\label{paper:note:2025-1101b}
\begin{enumerate}
\item[{\rm 1)}] \,\ $\widetilde{\eta}|_{ST^2S}(\tau,t) 
\, = \, 
e^{-\frac{\pi i}{3}} \, \widetilde{\eta}(\tau,t)$

\item[{\rm 2)}] \,\ $\widetilde{\theta}_{j,m}|_{ST^2S}(\tau,z,t) 
= 
\dfrac{-i}{2m} \sum\limits_{k \, \in \, \zzz/2m\zzz}
\bigg(\sum\limits_{\ell \, \in \, \zzz/2m\zzz}
e^{\frac{\pi i\ell(\ell-j-k)}{m}} \bigg) \, 
\widetilde{\theta}_{k,m}(\tau,z,t)$ 
\begin{equation}
\text{for} \quad m \in \nnn \quad \text{and} \quad j \in \zzz .
\label{paper:eqn:2025-1102b}
\end{equation}

\item[{\rm 3)}] \,\ 
$\widetilde{\vartheta}_{ab}|_{ST^2S}(\tau,z,t) 
\,\ = \,\
- \, i^{\, b+1} \, (-1)^{ab} \, 
\widetilde{\vartheta}_{ab}(\tau,z,t)$

{\rm i.e,}
$$\left\{
\begin{array}{lcr}
\widetilde{\vartheta}_{00}|_{ST^2S} &=& - \, i \, 
\widetilde{\vartheta}_{00}
\\[1mm]
\widetilde{\vartheta}_{01}|_{ST^2S} &=& \widetilde{\vartheta}_{01}
\end{array}\right. \hspace{10mm} \left\{
\begin{array}{lcr}
\widetilde{\vartheta}_{10}|_{ST^2S} &=& - \, i \, 
\widetilde{\vartheta}_{10}
\\[1mm]
\widetilde{\vartheta}_{11}|_{ST^2S} &=& - \, 
\widetilde{\vartheta}_{11}
\end{array}\right. 
$$
\end{enumerate}
\end{note}

\medskip

To simplify the formula \eqref{paper:eqn:2025-1102b} in the case 
$m \in 2\nnn$, we note the following: 

\medskip

\begin{note} \,\ 
\label{paper:note:2025-1101c}
For $m \, \in \, 2 \nnn$, the following formulas hold:

\begin{enumerate}
\item[{\rm 1)}] \quad $\sum\limits_{k \, \in \zzz/2m\zzz}
e^{\frac{\pi i}{m}k(k+1)} \,\ = \,\ 0 \hspace{10mm} 
{\rm i.e,} \hspace{10mm}
\sum\limits_{k \, \in \zzz/2m\zzz}
e^{\frac{\pi i}{m}(k+\frac12)^2} \,\ = \,\ 0$ 

\item[{\rm 2)}] \quad $\sum\limits_{k \, \in \zzz/2m\zzz}
e^{\frac{\pi i}{m}k(k+n)} \,\ = \,\ \left\{
\begin{array}{ccl}
0 & & {\rm if} \quad n \, \in \, \zzz_{\rm odd} 
\\[1.5mm]
e^{-\frac{\pi i}{4m}n^2} \gamma_m
& & {\rm if} \quad n \, \in \, \zzz_{\rm even}
\end{array}\right. $

\medskip

where 
\begin{equation}
\gamma_m \, := \,\ 
\sum_{k \, \in \zzz/2m\zzz} e^{\frac{\pi i}{m}k^2}
\label{paper:eqn:2025-1102c}
\end{equation}
\end{enumerate}
\end{note}

\begin{proof} 1) \quad $\sum\limits_{k \, \in \zzz/2m\zzz}
e^{\frac{\pi i}{m}k(k+1)}
\, = \, 
\underbrace{\sum\limits_{1 \, \leq \, k \, \leq \, m}
e^{\frac{\pi i}{m}k(k+1)}}_{(A)}
\,\ + \, 
\underbrace{\sum\limits_{-m \, < \, k \, \leq \, 0}
e^{\frac{\pi i}{m}k(k+1)}}_{(B)}$

\medskip

We compute $(B)$, by putting $k=k'-m$ :
$$
(B) \, = \hspace{-3mm}
\sum_{1 \, \leq \, k' \, \leq \, m}
\overbrace{e^{\frac{\pi i}{m}(k'-m)(k'+1-m)}}^{\substack{
{\displaystyle \hspace{-10mm}
e^{\frac{\pi i}{m}[k'(k'+1)-(2k'+1)m+m^2]}} \\[-0.5mm] ||
}}
\, = \, 
\sum_{1 \, \leq \, k \, \leq \, m}
e^{\frac{\pi i}{m}k(k+1)} \, 
\underbrace{e^{\pi i(2k+1)}}_{-1} \, \underbrace{e^{\pi im}}_{1}
\,\ = \,\ (-1) \times (A)
$$
So we have \, $(A)+(B)=0$, proving 1).  \,\ 
2) follows from 1) since 
$$
\sum_{k \, \in \zzz/2m\zzz} 
e^{\frac{\pi i}{m}k(k+n)} 
\, = 
\sum_{k \, \in \zzz/2m\zzz} 
e^{\frac{\pi i}{m}(k+\frac{n}{2})^2} 
e^{-\frac{\pi in^2}{4m}}
$$

\vspace{-6mm}

\end{proof}

\medskip

From the formula \eqref{paper:eqn:2025-1102b} and Note 
\ref{paper:note:2025-1101c}, we obtain the following:

\medskip


\begin{lemma} \quad 
\label{paoer:lemma:2025-1101a}
For $m \in 2\nnn$ and $j \in \zzz$, the following formulas hold:

\begin{enumerate}
\item[{\rm 1)}] \quad $\widetilde{\theta}_{j,m}|_{ST^2S} 
\, = \, 
\dfrac{-i \, \gamma_m}{2m} \, 
\sum\limits_{\substack{k \, \in \, \zzz/2m\zzz \\[1mm]
k \, \equiv \, j \, {\rm mod} \, 2}}
e^{-\frac{\pi i}{4m}(j+k)^2} \, \widetilde{\theta}_{k,m}$

\vspace{-14mm}

\begin{equation}
\label{paper:eqn:2025-1102d1}
\end{equation}

\vspace{3mm}

\item[{\rm 2)}] 
\begin{subequations}
\begin{enumerate}
\item[{\rm (i)}] \,\ $(
\widetilde{\theta}_{j,m}-\widetilde{\theta}_{-j,m})|_{ST^2S} 
\, = \, 
\dfrac{-i \, \gamma_m}{2m} \, 
\sum\limits_{\substack{k \, \in \, \zzz/2m\zzz \\[1mm]
k \, \equiv \, j \, {\rm mod} \, 2}}
e^{-\frac{\pi i}{4m}(j+k)^2} \, 
(\widetilde{\theta}_{k,m}-\widetilde{\theta}_{-k,m})$
\begin{equation}
= \,\ \frac{i \, \gamma_m}{2m} \, 
\sum\limits_{\substack{k \, \in \, \zzz/2m\zzz \\[1mm]
k \, \equiv \, j \, {\rm mod} \, 2}}
e^{-\frac{\pi i}{4m}(j-k)^2} \, 
(\widetilde{\theta}_{k,m}-\widetilde{\theta}_{-k,m})
\label{paper:eqn:2025-1102d2}
\end{equation}

\item[{\rm (ii)}] \,\ $(
\widetilde{\theta}_{j,m}-\widetilde{\theta}_{-j,m})|_{ST^2S} 
\, = \, 
\dfrac{- \, \gamma_m}{m} \hspace{-1mm}
\sum\limits_{\substack{0 \, < \, k \, < \, m \\[1mm]
k \, \equiv \, j \, {\rm mod} \, 2}}
e^{-\frac{\pi i}{4m}(j^2+k^2)} \, \sin \dfrac{\pi jk}{2m} \,\ 
(\widetilde{\theta}_{k,m}-\widetilde{\theta}_{-k,m})$

\vspace{-14mm}

\begin{equation}
\label{paper:eqn:2025-1102d3}
\end{equation}
\end{enumerate}
\end{subequations}
\end{enumerate}
\end{lemma}

\vspace{3mm}

\begin{proof} 1) \,\ By the formula \eqref{paper:eqn:2025-1102b}, 
we have
$$
\widetilde{\theta}_{j,m}|_{ST^2S} \, = \, 
\frac{-i}{2m} \, \sum_{k \, \in \, \zzz/2m\zzz}
\bigg(
\underbrace{\sum_{\ell \, \in \, \zzz/2m\zzz}
e^{\frac{\pi i\ell[\ell-(j+k)]}{m}}}_{A(k)} 
\bigg) \, \widetilde{\theta}_{k,m}
$$

\vspace{-3mm}

\noindent
where $A(k)$ becomes as follows by Note \ref{paper:note:2025-1101c}: 
$$
A(k) \,\ = \,\ \left\{
\begin{array}{lcl}
e^{-\frac{\pi i}{4m}(j+k)^2} \, \gamma_m
& & {\rm if} \quad j+k \, \in \, 2 \, \zzz 
\\[1mm]
0 & &{\rm if} \quad j+k \, \in \, \zzz_{\rm odd}
\end{array}\right.
$$
So we have
$$
\widetilde{\theta}_{j,m}|_{ST^2S} \, = \, 
\frac{-i \, \gamma_m}{2m} \, 
\sum_{\substack{k \, \in \, \zzz/2m\zzz \\[1mm]
k \, \equiv \, j \, {\rm mod} \, 2}}
e^{-\frac{\pi i}{4m}(j+k)^2} \, \widetilde{\theta}_{k,m}
$$
proving 1). \, The claims in 2) follow easily from 1).
\end{proof}

\subsection{Modular transformation of characters}
\label{paper:subsec:character:modular-transf}



\begin{prop} \,\ 
\label{paper:prop:2025-1105a}
For $\Lambda_{n_1,n_2}^{[m]} \, \in \, P^K_+(S^{\vee}_{(u=1)})$, 
the function $A_{n_1,n_2}^{\, [K]}$ and its modular transformation 
are given as follows:
\begin{enumerate}
\item[{\rm 1)}] $
\underset{\substack{\\[0.5mm] || \,\ put \\[0mm] 
{\displaystyle 
A_{n_1,n_2}^{\prime \, [K]}(\tau, z_1, z_2,t)
}}}{A_{n_1,n_2}^{\, [K]}(\tau, z_1, z_2,2t)}
= 
\big[\widetilde{\theta}_{n_1,2(K+3)}-
\widetilde{\theta}_{-n_1,2(K+3)}\big](\tau,z_1,t) \cdot
\big[
\widetilde{\theta}_{n_2,2(K+3)}-
\widetilde{\theta}_{-n_2,2(K+3)}\big](\tau,z_2,t)$

\vspace{-10mm}

\begin{equation}
\label{paper:eqn:2025-1105c1}
\end{equation}

\vspace{-2mm}

\item[{\rm 2)}]
\begin{subequations}
\begin{enumerate}
\item[{\rm (i)}]  \,\ $A_{n_1,n_2}^{\prime \, [K]}|_{ST^2S} 
(\tau, z_1, z_2,t)
\, = \,\ 
\dfrac{(\gamma_{2(K+3)})^2}{4(K+3)^2} \,\ 
e^{-\frac{\pi i}{8(K+3)}(n_1^2+n_2^2)}$
\begin{eqnarray} \hspace{-3mm}
\times 
\sum_{\substack{0 \, < \, k_1, \, k_2 \, < \, 2(K+3) \\[1mm]
k_1 \, \equiv \, n_1 \, {\rm mod} \, 2 \\[1mm]
k_2 \, \equiv \, n_2 \, {\rm mod} \, 2}} \hspace{-3mm}
e^{-\frac{\pi i}{8(K+3)}(k_1^2+k_2^2)} \, 
\Big(\sin \frac{\pi n_1k_1}{4(K+3)}\Big)
\Big(\sin \frac{\pi n_2k_2}{4(K+3)}\Big) \, 
A^{\prime \, [K]}_{k_1,k_2}(\tau, z_1,z_2,t)
\nonumber
\\[-8mm]
& &
\label{paper:eqn:2025-1105c2}
\end{eqnarray}

\vspace{2mm}

\item[{\rm (ii)}] \,\ $A_{n_1,n_2}^{\prime \, [K]}|_T 
(\tau, z_1, z_2,t)
\, = \,\ 
e^{\frac{\pi i}{4(K+3)}(n_1^2+n_2^2)} \, 
A_{n_1,n_2}^{\prime \, [K]}(\tau, z_1, z_2,t)$

\vspace{-12mm}

\begin{equation}
\label{paper:eqn:2025-1105c3}
\end{equation}
\end{enumerate}
\end{subequations}
\end{enumerate}
\end{prop}

\vspace{0mm}

\begin{proof} 1) is obtained by multiplying $e^{4\pi i(K+3)t}$ 
to both sides of \eqref{paper:eqn:2025-1031j}. \, 
The formulas in 2) are obtained by easy calculation using 
Notes \ref{paper:note:2025-1101a} and \ref{paper:note:2025-1101b}.
\end{proof}

\medskip


\begin{prop} \,\ 
\label{prop:2025-727a}
The denominator $R$ of $C_2^{(1)}$ and its modular transformation 
is given by the following formulas:

\begin{enumerate}
\item[{\rm 1)}] \,\ 

\vspace{-15mm}

\begin{subequations}
{\allowdisplaybreaks
\begin{eqnarray}
& & \hspace{-20mm}
R(\tau,z_1,z_2,t) \, = \,\ 
\frac{e^{6\pi it}}{\eta(\tau)^2} \, 
\vartheta_{11}(\tau, z_1) \, 
\vartheta_{11}(\tau, z_2) \, 
\vartheta_{11}(\tau, z_1-z_2) \, 
\vartheta_{11}(\tau, z_1+z_2)
\label{eqn:2025-727a1}
\\[2mm]
&=&
\frac{1}{\widetilde{\eta}(\tau,t)^2} \, 
\widetilde{\vartheta}_{11}(\tau, z_1,t) \, 
\widetilde{\vartheta}_{11}(\tau, z_2,t) \, 
\widetilde{\vartheta}_{11}(\tau, z_1-z_2,t) \, 
\widetilde{\vartheta}_{11}(\tau, z_1+z_2,t)
\label{eqn:2025-727a1}
\end{eqnarray}}
\end{subequations}

\vspace{-6mm}

\item[{\rm 2)}]
\begin{enumerate}
\item[{\rm (i)}] \,\ $R|_S(\tau,z_1,z_2,t) 
\hspace{5.5mm} = \,\ 
- \, i \, R(\tau,z_1,z_2,t)$

\vspace{1mm}

\item[{\rm (ii)}] \,\ $R|_T(\tau,z_1,z_2,t) 
\hspace{5.5mm} = \,\ 
e^{\frac{5\pi i}{6}} \, R(\tau,z_1,z_2,t)$

\vspace{1mm}

\item[{\rm (iii)}] \,\ $R|_{ST^2S}(\tau,z_1,z_2,t) 
\, = \,\ 
e^{\frac{2\pi i}{3}} \, R(\tau,z_1,z_2,t)$
\end{enumerate}
\end{enumerate}
\end{prop}

\begin{proof} 1) \,\ The denominator $R$ is computed by 
the following formula:
{\allowdisplaybreaks
\begin{eqnarray*}
& & \hspace{-7mm}
R \, = \,\ 
q^{\frac{|\rho|^2}{2h^{\vee}}} \, e^{\rho}
\prod\limits_{\alpha \in \Delta_+}(1-e^{-\alpha})^{{\rm mult} (\alpha)}
\\[2mm]
&=&
q^{\frac{|\rho|^2}{2h^{\vee}}} \, 
e^{\rho} \prod_{n=1}^{\infty}\bigg[(1-e^{-n\delta})^2
\prod_{\alpha \in \overline{\Delta}_+}
(1-e^{-(n-1)\delta-\alpha})(1-e^{-n\delta+\alpha})\bigg]
\\[0mm]
&=&
q^{\frac16 \cdot \frac52} 
e^{\rho} \prod_{n=1}^{\infty}\Big[
(1-e^{-n\delta})^2
(1-q^{n-1}e^{-\alpha_1})(1-q^ne^{\alpha_1})
(1-q^{n-1}e^{-\alpha_2})(1-q^ne^{\alpha_2})
\\[0mm]
& &
\times \,\ 
(1-q^{n-1}e^{-(\alpha_1+\alpha_2)})(1-q^ne^{(\alpha_1+\alpha_2)})
(1-q^{n-1}e^{-(2\alpha_1+\alpha_2)})(1-q^ne^{(2\alpha_1+\alpha_2)})
\Big]
\end{eqnarray*}}
In terms of coordinates defined by 
\eqref{paper:eqn:2025-1101h1} and \eqref{paper:eqn:2025-1101h2},
this formula is rewritten as follows:
{\allowdisplaybreaks
\begin{eqnarray*}
& & \hspace{-8mm}
R(\tau,z_1,z_2,t) \, = \,\
q^{\frac{5}{12}} \, \times \, 
e^{6\pi it} \, e^{2\pi i(\frac12z_1+\frac32z_2)} \, 
\varphi(q)^2 
\\[1mm]
& & \hspace{-5mm}
\times \,\ \prod_{n=1}^{\infty} \Big[
(1-q^{n-1}e^{-2\pi iz_1})(1-q^{n}e^{2\pi iz_1})
(1-q^{n-1}e^{-2\pi iz_2})(1-q^{n}e^{2\pi iz_2})
\\[2mm]
& &
\times \,\ 
(1-q^{n-1}e^{-2\pi i(z_1-z_2})(1-q^{n}e^{2\pi i(-z_1+z_2)})
(1-q^{n-1}e^{-2\pi i(z_1+z_2)})(1-q^{n}e^{2\pi i(z_1+z_2)})\Big]
\end{eqnarray*}}
where \,\ $\varphi(q) \, := \prod\limits_{n=1}^{\infty}(1-q^n)$. \,\ 
Noticing, by \eqref{paper:eqn:2025-1102a}, that 
$$
\prod_{n=1}^{\infty}
(1-q^{n-1}e^{-2\pi iz}) (1-q^{n}e^{2\pi iz}) 
\, = \,\ 
-i \, e^{-2\pi it} \, e^{-\pi iz} \, 
\frac{q^{-\frac18}}{\varphi(q)} \, \widetilde{\vartheta}_{11}(\tau,z,t)
$$
the above formula is rewritten as follows: 
{\allowdisplaybreaks
\begin{eqnarray*}
& & \hspace{-8mm}
R(\tau,z_1,z_2,t) = 
e^{-2\pi it}
q^{-\frac{1}{12}}
\frac{1}{\varphi(q)^2}
\widetilde{\vartheta}_{11}(\tau, z_1,t) \, 
\widetilde{\vartheta}_{11}(\tau, z_2,t) \, 
\widetilde{\vartheta}_{11}(\tau, z_1-z_2,t) \, 
\widetilde{\vartheta}_{11}(\tau, z_1+z_2,t) 
\\[0mm]
& & = \,\ 
\frac{1}{\widetilde{\eta}(\tau,t)^2} \, 
\widetilde{\vartheta}_{11}(\tau, z_1,t) \, 
\widetilde{\vartheta}_{11}(\tau, z_2,t) \, 
\widetilde{\vartheta}_{11}(\tau, z_1-z_2,t) \, 
\widetilde{\vartheta}_{11}(\tau, z_1+z_2,t) 
\end{eqnarray*}}
proving 1). \,\ 
The formulas in 2) are obtained by easy calculation using 
Notes \ref{paper:note:2025-1101a} and \ref{paper:note:2025-1101b}.
\end{proof}

\medskip

By Propositions \ref{paper:prop:2025-1105a} and \ref{prop:2025-727a},
the character of an admissible $C_2^{(1)}$-module
$L(\Lambda^{\, [K]}_{n_1,n_2})$ and its modular transformation 
are obtained as follows:

\medskip

\begin{prop} \,\ 
\label{paper:prop:2025-1105d}
For $\Lambda_{n_1,n_2}^{[K]} \, \in \, P^K_+(S^{\vee}_{(u=1)})$, 
the character of $L(\Lambda^{\, [K]}_{n_1,n_2})$ and its 
modular transformation are given by the following formulas:
\begin{enumerate}
\item[{\rm 1)}] \,\ ${\rm ch}_{L(\Lambda^{\, [K]}_{n_1,n_2})}
(\tau, z_1, z_2,t)
\, \overset{\substack{def \\[1mm]}}{:=} \,\ 
\dfrac{A_{n_1,n_2}^{\prime \, [K]}(\tau, z_1, z_2,t)}{
R(\tau, z_1, z_2,t)}$
$$
=
\frac{\eta(\tau,t)^2 \, 
\big[\widetilde{\theta}_{n_1,2(K+3)}-
\widetilde{\theta}_{-n_1,2(K+3)}\big](\tau,z_1,t) \cdot
\big[
\widetilde{\theta}_{n_2,2(K+3)}-
\widetilde{\theta}_{-n_2,2(K+3)}\big](\tau,z_2,t)}{
\widetilde{\vartheta}_{11}(\tau,z_1,t) \, 
\widetilde{\vartheta}_{11}(\tau,z_2,t) \, 
\widetilde{\vartheta}_{11}(\tau,z_1-z_2,t) \, 
\widetilde{\vartheta}_{11}(\tau,z_1+z_2,t)}
$$

\item[{\rm 2)}] \,\ ${\rm ch}_{L(\Lambda^{\, [K]}_{n_1,n_2})}|_{ST^2S} 
(\tau, z_1, z_2,t)
\, = \,\ - \, 
\dfrac{(\gamma_{2(K+3)})^2}{4(K+3)^2} \,\ 
e^{\frac{\pi i}{3}} \, 
e^{-\frac{\pi i}{8(K+3)}(n_1^2+n_2^2)}$
$$ \hspace{-8mm}
\times 
\sum_{\substack{0 \, < \, k_1, \, k_2 \, < \, 2(K+3) \\[1mm]
k_1 \, \equiv \, n_1 \, {\rm mod} \, 2 \\[1mm]
k_2 \, \equiv \, n_2 \, {\rm mod} \, 2}} \hspace{-3mm}
e^{-\frac{\pi i}{8(K+3)}(k_1^2+k_2^2)} \, 
\Big(\sin \frac{\pi n_1k_1}{4(K+3)}\Big)
\Big(\sin \frac{\pi n_2k_2}{4(K+3)}\Big) \, 
{\rm ch}_{L(\Lambda^{\, [K]}_{k_1,k_2})}
(\tau, z_1,z_2,t)
$$

\item[{\rm 3)}] \,\ ${\rm ch}_{L(\Lambda^{\, [K]}_{n_1,n_2})}|_T 
(\tau, z_1, z_2,t)
\, = \,\ 
- \, e^{\frac{\pi i}{6}} \, 
e^{\frac{\pi i}{4(K+3)}(n_1^2+n_2^2)} \, 
{\rm ch}_{L(\Lambda^{\, [K]}_{n_1,n_2})}(\tau, z_1, z_2,t)$
\end{enumerate}
\end{prop}

\section{Characters of quantum Hamiltonian reduction 
$\sim$ in general case}
\label{sec:QHR:characters}


Before going to the quantum Hamiltonian reduction (QHR) of an 
admissible $C_2^{(1)}$-module $L(\Lambda^{\, [K]}_{n_1,n_2})$, 
we recall the formulas for the character of QHR adding the 
parameter $t$ in general case of affine Lie algebras.

\medskip

Let $\ggg$ be an affine Lie algebra and $\overline{\ggg}$ be 
its underlying finite-dimensional simple Lie algebra of rank $\ell$
and $(x,e,f)$ be an $sl_2$-triple in $\overline{\ggg}$.
In this setting, we define the denominators 
$\overset{w}{R}{}^{(\pm)}(\tau,H,t)$ and 
$\overset{w}{R}{}^{(\ast)}(\tau,H,t)$ of the 
W-algebra $W(\ggg, f)$ by the following formulas: 
\begin{subequations}
{\allowdisplaybreaks
\begin{eqnarray}
& & \hspace{-10mm}
\overset{w}{R}{}^{(+)}(\tau,H,t)
:= \,\  
e^{\frac{\pi it}{2}(\ell+{\rm dim} \overline{\ggg}^f)}
\eta(\tau)^{\frac32\ell-\frac12{\rm dim} \overline{\ggg}^f}
\prod_{\alpha \in \overline{\Delta}_0^+}
\vartheta_{11}(\tau, \alpha(H))
\bigg[\prod_{\alpha \in \overline{\Delta}_{\frac12}}
\vartheta_{01}(\tau, \alpha(H))\bigg]^{\frac12}
\nonumber
\\[2mm]
& &
= \,\ 
\widetilde{\eta}(\tau,t)^{\frac32\ell-\frac12{\rm dim} \overline{\ggg}^f}
\prod_{\alpha \in \overline{\Delta}_0^+}
\widetilde{\vartheta}_{11}(\tau, \alpha(H),t)
\bigg[\prod_{\alpha \in \overline{\Delta}_{\frac12}}
\widetilde{\vartheta}_{01}(\tau, \alpha(H),t)\bigg]^{\frac12}
\label{paper:eqn:2025-1105a1}
\\[2mm]
& & \hspace{-10mm}
\overset{w}{R}{}^{(-)}(\tau,H,t)
:= \,\  
e^{\frac{\pi it}{2}(\ell+{\rm dim} \overline{\ggg}^f)}
\eta(\tau)^{\frac32\ell-\frac12{\rm dim} \overline{\ggg}^f}
\prod_{\alpha \in \overline{\Delta}_0^+}
\vartheta_{11}(\tau, \alpha(H))
\bigg[\prod_{\alpha \in \overline{\Delta}_{\frac12}}
\vartheta_{00}(\tau, \alpha(H))\bigg]^{\frac12}
\nonumber
\\[2mm]
& &
= \,\ 
\widetilde{\eta}(\tau,t)^{\frac32\ell-\frac12{\rm dim} \overline{\ggg}^f}
\prod_{\alpha \in \overline{\Delta}_0^+}
\widetilde{\vartheta}_{11}(\tau, \alpha(H),t)
\bigg[\prod_{\alpha \in \overline{\Delta}_{\frac12}}
\widetilde{\vartheta}_{00}(\tau, \alpha(H),t)\bigg]^{\frac12}
\label{paper:eqn:2025-1105a2}
\\[2mm]
& & \hspace{-10mm}
\overset{w}{R}{}^{(\ast)}(\tau,H,t)
:= \,\  
e^{\frac{\pi it}{2}(\ell+{\rm dim} \overline{\ggg}^f)}
\eta(\tau)^{\frac32\ell-\frac12{\rm dim} \overline{\ggg}^f}
\prod_{\alpha \in \overline{\Delta}_0^+}
\vartheta_{11}(\tau, \alpha(H))
\bigg[\prod_{\alpha \in \overline{\Delta}_{\frac12}}
\vartheta_{10}(\tau, \alpha(H))\bigg]^{\frac12}
\nonumber
\\[2mm]
& &
= \,\ 
\widetilde{\eta}(\tau,t)^{\frac32\ell-\frac12{\rm dim} \overline{\ggg}^f}
\prod_{\alpha \in \overline{\Delta}_0^+}
\widetilde{\vartheta}_{11}(\tau, \alpha(H),t)
\bigg[\prod_{\alpha \in \overline{\Delta}_{\frac12}}
\widetilde{\vartheta}_{10}(\tau, \alpha(H),t)\bigg]^{\frac12}
\label{paper:eqn:2025-1105a3}
\end{eqnarray}}
\end{subequations}
where \,\ $H \in \overline{\hhh}^f$ \, and \, 
$\overline{\Delta}_j := \, 
\{\alpha \in \overline{\Delta} \,\ ; \,\ \alpha(x)=j\}$ 
\, and \, 
$\overline{\Delta}_0^+ := \, 
\overline{\Delta}_0 \cap \overline{\Delta}^+$.

\medskip

Modular transformation of these functions are obtained 
by easy calculation using Notes \ref{paper:note:2025-1101a} 
and \ref{paper:note:2025-1101b} as follows:

\medskip

\begin{lemma} \,\ 
\label{paper:lemma:2025-1105a}
\begin{enumerate}
\item[{\rm 1)}]
\begin{enumerate}
\item[{\rm (i)}] \,\ $\overset{w}{R}{}^{(+)}|_S(\tau,H,t)
\, = \,\ 
e^{-\frac{\pi i}{4} \, {\rm dim} \, \overline{\ggg}_0} \,\ 
\overset{w}{R}{}^{(\ast)}(\tau,H,t)$

\item[{\rm (ii)}] \,\ $\overset{w}{R}{}^{(-)}|_S(\tau,H,t)
\, = \,\ 
e^{-\frac{\pi i}{4} \, {\rm dim} \, \overline{\ggg}_0} \,\ 
\overset{w}{R}{}^{(-)}(\tau,H,t)$

\item[{\rm (iii)}] \,\ $\overset{w}{R}{}^{(\ast)}|_S(\tau,H,t)
\, = \,\ 
e^{-\frac{\pi i}{4} \, {\rm dim} \, \overline{\ggg}_0} \,\ 
\overset{w}{R}{}^{(+)}(\tau,H,t)$
\end{enumerate}

\item[{\rm 2)}]
\begin{enumerate}
\item[{\rm (i)}] \,\ $\overset{w}{R}{}^{(+)}|_T(\tau,H,t)
\, = \, 
e^{\frac{\pi i}{24}(2{\rm dim} \overline{\ggg}_0
-{\rm dim} \overline{\ggg}_{\frac12})} \, 
\overset{w}{R}{}^{(-)}(\tau,H,t)
\, = \, 
e^{\frac{\pi i}{24}(3{\rm dim} \overline{\ggg}_0
-{\rm dim} \overline{\ggg}^f)} \, 
\overset{w}{R}{}^{(-)}(\tau,H,t)$

\item[{\rm (ii)}] \,\ $\overset{w}{R}{}^{(-)}|_T(\tau,H,t)
\, = \, 
e^{\frac{\pi i}{24}(2{\rm dim} \overline{\ggg}_0
-{\rm dim} \overline{\ggg}_{\frac12})} \, 
\overset{w}{R}{}^{(+)}(\tau,H,t)
\, = \, 
e^{\frac{\pi i}{24}(3{\rm dim} \overline{\ggg}_0
-{\rm dim} \overline{\ggg}^f)} \, 
\overset{w}{R}{}^{(+)}(\tau,H,t)$

\item[{\rm (iii)}] \,\ $\overset{w}{R}{}^{(\ast)}|_T(\tau,H,t)
\, = \,\ 
e^{\frac{\pi i}{12} {\rm dim} \overline{\ggg}^f} \, 
\overset{w}{R}{}^{(\ast)}(\tau,H,t)$
\end{enumerate}

\item[{\rm 3)}]
\begin{enumerate}
\item[{\rm (i)}] \,\ $\overset{w}{R}{}^{(+)}|_{ST^2S}(\tau,H,t)
\,\ = \,\ 
e^{\frac{\pi i}{6}({\rm dim} \overline{\ggg}^f
-3 \, {\rm dim} \overline{\ggg}_0)} \, 
\overset{w}{R}{}^{(+)}(\tau, H,t)$

\item[{\rm (ii)}] \,\ $\overset{w}{R}{}^{(-)}|_{ST^2S}(\tau,H,t)
\,\ = \,\ 
e^{-\frac{\pi i}{12}({\rm dim} \overline{\ggg}^f
+3 \, {\rm dim} \overline{\ggg}_0)} \, 
\overset{w}{R}{}^{(-)}(\tau,H,t)$
\end{enumerate}
\end{enumerate}
\end{lemma}

\medskip %

For an admissible $\ggg$-module $L(\Lambda)$, we define the 
characters $\overset{w}{\rm ch}{}^{(\pm)}_{H(\Lambda)}$ and 
$\overset{w}{\rm ch}{}^{(\ast)}_{H(\Lambda)}$ of the 
quantum Hamiltonian reduction by the following formulas:
\begin{subequations}
{\allowdisplaybreaks
\begin{eqnarray}
\big[\overset{w}{R}{}^{(+)} \cdot 
\overset{w}{\rm ch}{}^{(+)}_{H(\Lambda)}\big](\tau,H,t)
&:=& 
A_{\Lambda+\rho}\Big(\tau, \, H-\tau x, \, 2t+\frac{\tau}{2}|x|^2\Big)
\label{paper:eqn:2025-1105b1}
\\[2mm]
\big[\overset{w}{R}{}^{(-)} \cdot 
\overset{w}{\rm ch}{}^{(-)}_{H_f(\Lambda)}\big](\tau,H,t)
&:=& 
A_{\Lambda+\rho}\Big(\tau, \, H-\tau x+x, \, 2t+\frac{\tau}{2}|x|^2\Big)
\nonumber
\\[2mm]
&=&
e^{4\pi i(\Lambda+\rho|x)} \, 
A_{\Lambda+\rho}\Big(\tau, \, H-\tau x-x, \, 2t+\frac{\tau}{2}|x|^2\Big)
\label{paper:eqn:2025-1105b2}
\\[2mm]
\big[\overset{w}{R}{}^{(\ast)} \cdot 
\overset{w}{\rm ch}{}^{(\ast)}_{H(\Lambda)}\big](\tau, \, H, \, t)
&:=& 
A_{\Lambda+\rho}(\tau, \, H, \, 2t)
\label{paper:eqn:2025-1105b3}
\end{eqnarray}}
\end{subequations}
where \,\ $H \in \overline{\hhh}^f$ and $
A_{\Lambda+\rho} \, := \, 
q^{\frac{|\overline{\Lambda+\rho}|^2}{2(K+h^{\vee})}}
\sum\limits_{w \in W_{\Lambda}}\varepsilon(w) e^{w(\Lambda+\rho)}$
and $W_{\Lambda}$ is the Weyl group generated by reflections 
$\{r_{\alpha} \,\ ; \,\ \alpha \in \Delta^{\vee re}_{\Lambda}\}$.

\section{Quantum Hamiltonian reduction of admissible 
$C_2^{(1)}$-modules in the case $u=1$ w.r.to the minimal 
nilpotent orbit}
\label{sec:QHR:minimal}

\subsection{Functions $f_{j,m}^{(\pm)}(\tau,z,t)$}
\label{subsec:functions:f}

In order to compute the modular transformation properties of 
characters of the quantum Hamiltonian reduction, 
we define the functions $f_{j,m}^{(\pm)}$, for 
$m \in 2 \nnn$ and $j \in \zzz$, by 
\begin{equation}
f_{j,m}^{(\pm)}(\tau,z,t) \, := \,\
\big[\widetilde{\theta}_{j,m} \pm \widetilde{\theta}_{j+m,m}\big](\tau,z,t)
\label{eqn:2025-721d}
\end{equation}
The modular transformation of these functions are obtained easily 
from those of $\widetilde{\theta}_{j,m}$  described in Notes 
\ref{paper:note:2025-1101a} and \ref{paper:note:2025-1101b}
as follows.

\medskip

\begin{lemma} \,\ 
\label{lemma:2025-721b} 
\begin{enumerate}
\item[{\rm 1)}]
\begin{enumerate}
\item[{\rm (i)}] \quad $f_{j,m}^{(\pm)}|_{ST^2S} 
\,\ = \,\ 
\dfrac{-i \, \gamma_m}{2m} \, 
\sum\limits_{\substack{k \, \in \, \zzz/2m\zzz \\[1mm]
k \, \equiv \, j \, {\rm mod} \, 2}}
e^{-\frac{\pi i}{4m}(j+k)^2} \, f_{k,m}^{(\pm)}$

\vspace{1mm}

\item[{\rm (ii)}] \quad $f_{j,m}^{(\pm)}|_{T} 
\,\ = \,\ e^{\frac{\pi ij^2}{2m}} \, \times \, \left\{
\begin{array}{ccl}
f_{j,m}^{(\pm)} & & {\rm if} \quad j+\frac{m}{2} \, = \, {\rm even}
\\[1.5mm]
f_{j,m}^{(\mp)} & & {\rm if} \quad j+\frac{m}{2} \, = \, {\rm odd}
\end{array}\right. $
\end{enumerate}

\vspace{1mm}

\item[{\rm 2)}]
\begin{subequations}
\begin{enumerate}
\item[{\rm (i)}] \quad $
(f_{j,m}^{(\pm)} \, + \, f_{-j,m}^{(\pm)})|_{ST^2S} 
\,\ = \,\ 
\dfrac{-i \, \gamma_m}{2m} \, 
\sum\limits_{\substack{k \, \in \, \zzz/2m\zzz \\[1mm]
k \, \equiv \, j \, {\rm mod} \, 2}}
e^{-\frac{\pi i}{4m}(j+k)^2} \, 
(f_{k,m}^{(\pm)} \, + \, f_{-k,m}^{(\pm)})$

\vspace{-14mm}

\begin{equation}
\label{paper:eqn:2025-1105e1}
\end{equation}

\vspace{4mm}

\item[{\rm (ii)}] \quad $
(f_{j,m}^{(\pm)} \, - \, f_{-j,m}^{(\pm)})|_{ST^2S} 
\,\ = \,\ 
\dfrac{-i \, \gamma_m}{2m} \, 
\sum\limits_{\substack{k \, \in \, \zzz/2m\zzz \\[1mm]
k \, \equiv \, j \, {\rm mod} \, 2}}
e^{-\frac{\pi i}{4m}(j+k)^2} \, 
(f_{k,m}^{(\pm)} \, - \, f_{-k,m}^{(\pm)})$

\vspace{-14mm}

\begin{equation}
\label{paper:eqn:2025-1105e2}
\end{equation}
\end{enumerate}
\end{subequations}

\vspace{4mm}

\item[{\rm 3)}]
\begin{subequations}
\begin{enumerate}
\item[{\rm (i)}] \quad $
(f_{j,m}^{(\pm)} \, + \, f_{-j,m}^{(\pm)})|_{T} 
\,\ = \,\ e^{\frac{\pi ij^2}{2m}} \, \times \, \left\{
\begin{array}{ccl}
f_{j,m}^{(\pm)} \, + \, f_{-j,m}^{(\pm)}
& & {\rm if} \quad j+\frac{m}{2} \, = \, {\rm even}
\\[2mm]
f_{j,m}^{(\mp)} \, + \, f_{-j,m}^{(\mp)} 
& & {\rm if} \quad j+\frac{m}{2} \, = \, {\rm odd}
\end{array}\right. $

\vspace{-11.5mm}

\begin{equation}
\label{paper:eqn:2025-1105f1}
\end{equation}

\vspace{4mm}

\item[{\rm (ii)}] \quad $
(f_{j,m}^{(\pm)} \, - \, f_{-j,m}^{(\pm)})|_{T} 
\,\ = \,\ e^{\frac{\pi ij^2}{2m}} \, \times \, \left\{
\begin{array}{ccl}
f_{j,m}^{(\pm)} \, - \, f_{-j,m}^{(\pm)}
& & {\rm if} \quad j+\frac{m}{2} \, = \, {\rm even}
\\[2mm]
f_{j,m}^{(\mp)} \, - \, f_{-j,m}^{(\mp)} 
& & {\rm if} \quad j+\frac{m}{2} \, = \, {\rm odd}
\end{array}\right. $

\vspace{-11.5mm}

\begin{equation}
\label{paper:eqn:2025-1105f2}
\end{equation}
\end{enumerate}
\end{subequations}
\end{enumerate}
\end{lemma}

\vspace{5mm}

From the definition \eqref{eqn:2025-721d}, it is very easy 
to see the following Notes:

\medskip

\begin{note} \,\
\label{note:2025-717d}
\label{note:2025-717a}
For $m \in 2\nnn$ and $j \in \zzz$, the following formulas hold:

\begin{enumerate}
\item[{\rm 1)}] \,\ $f_{j+m,m}^{(\pm)} \, = \, \pm \, f_{j,m}^{(\pm)}$

\item[{\rm 2)}] \,\ $f_{\frac{m}{2},m}^{(\pm)}
\, \mp \, f_{-\frac{m}{2},m}^{(\pm)} \, = \, 0$
\end{enumerate}
\end{note}

\medskip

\begin{note} \,\ 
\label{note:2025-717b}
Let $m \, \in \, 2 \, \nnn$. For $j \in \zzz$ such 
that \, $\frac{m}{2} \, < \, j \, \leq \, \frac{m}{2}+m$, 
we put \, $j' \, := \, j-m$. \, Then

\begin{enumerate}
\item[{\rm 1)}] \quad $-\frac{m}{2} \, < \, j' \, \leq \, \frac{m}{2}$

\vspace{1.5mm}

\item[{\rm 2)}] \quad $f^{(\pm)}_{j,m} \,\ = \,\ \pm \, f^{(\pm)}_{j',m}$

\vspace{1.5mm}

\item[{\rm 3)}] \quad $f^{(\pm)}_{j,m} \, + \, f^{(\pm)}_{-j,m}
\,\ = \,\ 
\pm \, (f^{(\pm)}_{j',m} \, + \, f^{(\pm)}_{-j',m})$

\vspace{1.5mm}

\item[{\rm 4)}] \quad $f^{(\pm)}_{j,m} \, - \, f^{(\pm)}_{-j,m}
\,\ = \,\ 
\pm \, (f^{(\pm)}_{j',m} \, - \, f^{(\pm)}_{-j',m})$
\end{enumerate}
\end{note}

\medskip

Using Note \ref{note:2025-717b}, formulas 
\eqref{paper:eqn:2025-1105e1} and \eqref{paper:eqn:2025-1105e2} 
in Lemma \ref{lemma:2025-721b} are rewritten as follows:

\medskip

\begin{lemma} \,\ 
\label{lemma:2025-721d}
For $m \in 2\nnn$ and $j \in \zzz$, the following formulas hold:
\begin{subequations}
\begin{enumerate}
\item[{\rm 1)}] $(f_{j,m}^{(\pm)} + 
f_{-j,m}^{(\pm)})|_{ST^2S} 
=
\dfrac{-i \, \gamma_m}{2m} \hspace{-2mm} 
\sum\limits_{\substack{\\[0.5mm] 
-\frac{m}{2} \, < \, k \, \leq \frac{m}{2} 
\\[1mm] k \, \equiv \, j \, {\rm mod} \, 2}} \hspace{-2mm}
e^{-\frac{\pi i}{4m}(j+k)^2} \, \big\{
1 \pm (-1)^{\frac{j+k}{2}} \, e^{-\frac{\pi im}{4}}\big\} \,
(f_{k,m}^{(\pm)} \, + \, f_{-k,m}^{(\pm)}) $

\vspace{-10.5mm}

\begin{equation}
\label{paper:eqn:2025-1105g1}
\end{equation}

\vspace{1mm}

\item[{\rm 2)}] $(f_{j,m}^{(\pm)} - 
f_{-j,m}^{(\pm)})|_{ST^2S} 
=
\dfrac{-i \, \gamma_m}{2m} \hspace{-2mm} 
\sum\limits_{\substack{\\[0.5mm] 
-\frac{m}{2} \, < \, k \, \leq \frac{m}{2} 
\\[1mm] k \, \equiv \, j \, {\rm mod} \, 2}} \hspace{-2mm}
e^{-\frac{\pi i}{4m}(j+k)^2} \, \big\{
1 \pm (-1)^{\frac{j+k}{2}} \, e^{-\frac{\pi im}{4}}\big\} \,
(f_{k,m}^{(\pm)} \, - \, f_{-k,m}^{(\pm)}) $

\vspace{-10.5mm}

\begin{equation}
\label{paper:eqn:2025-1105g2}
\end{equation}
\end{enumerate}
\end{subequations}
\end{lemma}

\vspace{1mm}

\begin{proof} 1) \,\ We compute the RHS of 
\eqref{paper:eqn:2025-1105e1} as follows:
{\allowdisplaybreaks
\begin{eqnarray*}
& & \hspace{-10mm}
(f_{j,m}^{(\pm)} \, + \, f_{-j,m}^{(\pm)})|_{ST^2S} 
\,\ = \,\ 
\dfrac{-i \, \gamma_m}{2m} \, 
\sum\limits_{\substack{
-\frac{m}{2} \, < \, k \, \leq \frac{m}{2}+m \\[1mm]
k \, \equiv \, j \, {\rm mod} \, 2}}
e^{-\frac{\pi i}{4m}(j+k)^2} \, 
(f_{k,m}^{(\pm)} \, + \, f_{-k,m}^{(\pm)})
\\[2mm]
&=&
\frac{-i \, \gamma_m}{2m} \, \bigg[
\underbrace{
\sum_{\substack{-\frac{m}{2} \, < \, k \, \leq \frac{m}{2} \\[1mm]
k \, \equiv \, j \, {\rm mod} \, 2}}}_{(A)}
\, + \, 
\underbrace{\sum_{\substack{
\frac{m}{2} \, < \, k \, \leq \, \frac{m}{2}+m \\[1mm]
k \, \equiv \, j \, {\rm mod} \, 2}}}_{(B)}
\bigg] \, 
e^{-\frac{\pi i}{4m}(j+k)^2} \, 
(f_{k,m}^{(\pm)} \, + \, f_{-k,m}^{(\pm)})
\end{eqnarray*}}
Putting $k':=k-m$ and using Note \ref{note:2025-717b}, the part 
$(B)$ is rewitten as follows:
$$
(B) = \pm \, \frac{-i \, \gamma_m}{2m} \, 
\sum_{\substack{-\frac{m}{2} \, < \, k' \, \leq \frac{m}{2} \\[1mm]
k' \, \equiv \, j \, {\rm mod} \, 2}}
e^{-\frac{\pi i}{4m}(j+k'+m)^2} \, 
(f_{k',m}^{(\pm)} \, + \, f_{-k',m}^{(\pm)})
$$
So we have 
{\allowdisplaybreaks
\begin{eqnarray*}
(A)+(B) &=&
\frac{-i \, \gamma_m}{2m} \, 
\sum_{\substack{-\frac{m}{2} \, < \, k \, \leq \frac{m}{2} 
\\[1mm] k \, \equiv \, j \, {\rm mod} \, 2}}
\big(e^{-\frac{\pi i}{4m}(j+k)^2} \pm \hspace{-10mm}
\underbrace{e^{-\frac{\pi i}{4m}(j+k+m)^2}}_{\substack{|| 
\\[-1mm] {\displaystyle \hspace{4mm}
e^{-\frac{\pi i}{4m}(j+k)^2} \hspace{-2mm}
\underbrace{e^{-\frac{\pi i}{2}(j+k)}}_{\substack{|| \\[-1mm] 
{\displaystyle \hspace{7mm}
(-1)^{\frac{j+k}{2}}
}}} \hspace{-2mm}
e^{-\frac{\pi im}{4}}
}}} \hspace{-10mm}
\big) \, 
(f_{k,m}^{(\pm)} \, + \, f_{-k,m}^{(\pm)})
\\[2mm]
&=&
\frac{-i \, \gamma_m}{2m} \, 
\sum_{\substack{-\frac{m}{2} \, < \, k \, \leq \, \frac{m}{2} 
\\[1mm] k \, \equiv \, j \, {\rm mod} \, 2}}
e^{-\frac{\pi i}{4m}(j+k)^2} \, \big\{
1 \pm (-1)^{\frac{j+k}{2}} \, e^{-\frac{\pi im}{4}}\big\} \,
(f_{k,m}^{(\pm)} \, + \, f_{-k,m}^{(\pm)}) 
\end{eqnarray*}}
proving \eqref{paper:eqn:2025-1105g1}. \, 
The proof of \eqref{paper:eqn:2025-1105g2} is quite similar.
\end{proof}

\subsection{Characters of quantum Hamiltonian reduction}
\label{subsec:QHR:characters:minimal}


In this section we compute the characters of quantum Hamiltonian 
reduction of admissible $C^{(1)}_2$-modules with respect to 
the minimal nolpotent element $f=e_{-\theta}$ 
and their modular transformation properties.

\medskip

\begin{prop} \,\ 
\label{prop:2025-720a}
For $\Lambda_{n_1,n_2}^{[K]} \in P^+_K(S^{\vee}_{(u=1)})$, 
the characters 
$\overset{w}{\rm ch}{}^{(\pm)}_{H(\Lambda_{n_1,n_2}^{[K]})}$
of quantum Hamiltonian reduction with respect to the $sl_2$-triple 
$(x=\frac12 \theta, e=\frac12 e_{\theta}, f=e_{-\theta})$, 
where $\theta$ is the highest root, of $L(\Lambda_{n_1,n_2}^{[K]})$
are given by the following formulas:

\begin{enumerate}
\item[{\rm 1)}] \, $\big[\overset{w}{R}{}^{(+)} \cdot \, 
\overset{w}{\rm ch}{}^{(+)}_{H(\Lambda_{n_1,n_2}^{[K]})}\big]
(\tau,z\alpha_2,t) \, = 
\big[\widetilde{\theta}_{n_1+(K+3),2(K+3)}-
\widetilde{\theta}_{-n_1+(K+3),2(K+3)}\big](\tau,z,t)$
$$
\times \,\ \big[
\widetilde{\theta}_{n_2-(K+3),2(K+3)}-
\widetilde{\theta}_{-n_2-(K+3),2(K+3)}\big](\tau,z,t)
$$
\item[{\rm 2)}] \, $\big[\overset{w}{R}{}^{(-)} \cdot \, 
\overset{w}{\rm ch}{}^{(-)}_{H(\Lambda_{n_1,n_2}^{[K]})}\big]
(\tau,z\alpha_2,t)$
{\allowdisplaybreaks
\begin{eqnarray*}
&=&
- \,\ e^{\frac{\pi i}{2}(n_2-n_1)} \, 
\big[\widetilde{\theta}_{n_1+(K+3),2(K+3)}-e^{\pi in_1} \, 
\widetilde{\theta}_{-n_1+(K+3),2(K+3)}\big](\tau,z,t)
\\[2mm]
& & \hspace{15mm}
\times \,\ \big[
\widetilde{\theta}_{n_2-(K+3),2(K+3)}-e^{-\pi in_2} \, 
\widetilde{\theta}_{-n_2-(K+3),2(K+3)}\big](\tau,z,t)
\end{eqnarray*}}
\end{enumerate}
\end{prop}

\begin{proof} \,\ For $x= \frac12 \theta=\frac12(2\alpha_1+\alpha_2)$
and $H = z\alpha_2 \in \overline{\hhh}^f$ \,\ $(z \in \ccc)$, we have  
\begin{equation}
\left\{
\begin{array}{lclcc}
H-\tau x &=& z\alpha_2-\frac{\tau}{2}(2\alpha_1+\alpha_2)
&=&
(-z-\frac{\tau}{2})\alpha_1+(z-\frac{\tau}{2})(\alpha_1+\alpha_2)
\\[2mm]
H-\tau x+x &=& z\alpha_2+(\frac12-\frac{\tau}{2})(2\alpha_1+\alpha_2)
&=&
(-z+\frac12-\frac{\tau}{2})\alpha_1
+(z+\frac12-\frac{\tau}{2})(\alpha_1+\alpha_2)
\end{array}\right.
\label{eqn:2025-720a}
\end{equation}
Then the numerators of the quantum reduction are given 
by \eqref{paper:eqn:2025-1105b1} and \eqref{paper:eqn:2025-1105b2}
and \eqref{paper:eqn:2025-1105c1} as follows:
\begin{subequations}
{\allowdisplaybreaks
\begin{eqnarray}
& & \hspace{-20mm}
\big[\overset{K}{R}{}^{(+)} \cdot 
\overset{w}{\rm ch}{}^{(+)}_{H(\Lambda_{n_1,n_2}^{\, [K]})}
\big](\tau,H,t)
\, = \, 
A_{n_1,n_2}^{\, [K]}\Big(\tau, \, -\tau x+H, \, 
2t+\frac{|x|^2}{2}\tau\Big)
\nonumber
\\[2mm]
&=&
A_{n_1,n_2}^{\, [K]}\Big(\tau, \, 
\Big(-z-\frac{\tau}{2}\Big)\alpha_1+
\Big(z-\frac{\tau}{2}\Big)(\alpha_1+\alpha_2),
\, 2t+\frac{\tau}{4}\Big)
\nonumber
\\[2mm]
&=&
e^{\frac{\pi i(K+3)\tau}{2}} A_{n_1,n_2}^{\, [K]}
\Big(\tau, \,\ 
-z-\frac{\tau}{2}, \,\ z-\frac{\tau}{2}, \,\ 2t\Big)
\nonumber
\\[2mm]
&=&
q^{\frac{K+3}{4}} \, 
\big[\widetilde{\theta}_{n_1,2(K+3)}-
\widetilde{\theta}_{-n_1,2(K+3)}\big]
\Big(\tau, \, -z-\frac{\tau}{2}, \, t\Big) 
\nonumber
\\[1mm]
& & \hspace{4mm}
\times \,\ 
\big[\widetilde{\theta}_{n_2,2(K+3)}-
\widetilde{\theta}_{-n_2,2(K+3)}\big]
\Big(\tau, \, z-\frac{\tau}{2}, \, t\Big)
\label{eqn:2025-720b1}
\\[2mm]
& & \hspace{-20mm}
\big[\overset{w}{R}{}^{(-)} \cdot 
\overset{w}{\rm ch}{}^{(-)}_{H(\Lambda_{(n_1,n_2)}^{\, [K]})}
\big](\tau,H,t)
\, = \, 
A_{n_1,n_2}^{\, [K]}\Big(\tau, \, x-\tau x+H, \, 2t+
\frac{|x|^2}{2}\tau\Big)
\nonumber
\\[2mm]
&=&
A_{n_1,n_2}^{\, [K]}\Big(\tau, \, 
\Big(-z+\frac12-\frac{\tau}{2}\Big)\alpha_1
+\Big(z+\frac12-\frac{\tau}{2}\Big)(\alpha_1+\alpha_2),
\, 2t+\frac{\tau}{4}\Big)
\nonumber
\\[2mm]
&=&
e^{\frac{\pi i(K+3)\tau}{2}} A_{n_1,n_2}^{\, [K]}
\Big(\tau, \,\ 
-z+\frac12-\frac{\tau}{2}, \,\ z+\frac12-\frac{\tau}{2}, \,\ 2t\Big)
\nonumber
\\[2mm]
&=&
q^{\frac{K+3}{4}} \, 
\big[\widetilde{\theta}_{n_1,2(K+3)}-
\widetilde{\theta}_{-n_1,2(K+3)}\big]
\Big(\tau, \, -z+\frac12-\frac{\tau}{2}, \, t\Big) 
\nonumber
\\[1mm]
& & \hspace{4mm}
\times \,\ 
\big[\widetilde{\theta}_{n_2,2(K+3)}-
\widetilde{\theta}_{-n_2,2(K+3)}\big]
\Big(\tau, \, z+\frac12-\frac{\tau}{2}, \, t\Big)
\label{eqn:2025-720b2}
\end{eqnarray}}
\end{subequations}

By the elliptic transformation properties of $\theta_{j,m}(\tau,z)$:
$$ \left\{
\begin{array}{lcccl}
\theta_{j,m}(\tau, z+a) &=&
e^{\pi ija} \, \theta_{j,m}(\tau, z) & &
({\rm if} \,\ am \, \in \, \zzz)
\\[2mm]
\theta_{j,m}(\tau, z+a\tau) &=&
q^{-\frac{ma^2}{4}} \, e^{-\pi imaz} \, 
\theta_{j+am,m}(\tau, z)
& & (a \, \in \, \rrr)
\end{array}\right.
$$
we have
\begin{equation}
\left\{
\begin{array}{lcr}
\widetilde{\theta}_{j,2(K+3)}(\tau, \, z-\frac{\tau}{2}, \, t) 
&=&
q^{-\frac{K+3}{8}} \, e^{\pi i(K+3)z} \, 
\widetilde{\theta}_{j-(K+3),2(m+3)}(\tau,z,t)
\\[2mm]
\widetilde{\theta}_{j,2(K+3)}(\tau, \, z+\frac12-\frac{\tau}{2}, \, t) 
&=&
e^{\frac{\pi ij}{2}} \, 
q^{-\frac{K+3}{8}} \, e^{\pi i(K+3)z} \, 
\widetilde{\theta}_{j-(K+3),2(K+3)}(\tau,z,t)
\end{array}\right.
\label{eqn:2025-720c}
\end{equation}
Using \eqref{eqn:2025-720c}, the above formulas 
\eqref{eqn:2025-720b1} and \eqref{eqn:2025-720b2} are rewritten
as follows:
{\allowdisplaybreaks
\begin{eqnarray*}
& & \hspace{-10mm}
\big[\overset{w}{R}{}^{(+)} \cdot 
\overset{w}{\rm ch}{}^{(+)}_{H(\Lambda_{n_1,n_2}^{\, [K]})}
\big](\tau,z,t)
= \, 
e^{-\pi i(K+3)z} \big[
\widetilde{\theta}_{n_1-(K+3),2(K+3)}-
\widetilde{\theta}_{-n_1-(K+3),2(K+3)}\big](\tau,-z,t)
\\[2mm]
& & \hspace{15mm}
\times \,\ 
e^{\pi i(K+3)z} \big[
\widetilde{\theta}_{n_2-(K+3),2(K+3)}-
\widetilde{\theta}_{-n_2-(K+3),2(K+3)}\big](\tau,z,t)
\\[2mm]
&=&
- \,\ 
\big[\widetilde{\theta}_{n_1+(K+3),2(K+3)}-
\widetilde{\theta}_{-n_1+(K+3),2(K+3)}\big](\tau,z,t)
\\[2mm]
& & \hspace{5mm}
\times \,\ 
\big[\widetilde{\theta}_{n_2-(K+3),2(K+3)}-
\widetilde{\theta}_{-n_2-(K+3),2(K+3)}\big](\tau,z,t)
\\[3mm]
& & \hspace{-10mm}
\big[\overset{w}{R}{}^{(-)} \cdot 
\overset{w}{\rm ch}{}^{(-)}_{H(\Lambda_{n_1,n_2}^{\, [K]})}
\big](\tau,z,t)
\\[2mm]
&=&
e^{-\pi i(K+3)z} \big[
e^{\frac{\pi in_1}{2}} \, \widetilde{\theta}_{n_1-(K+3),2(K+3)}
-
e^{-\frac{\pi in_1}{2}} \, 
\widetilde{\theta}_{-n_1-(K+3),2(K+3)}\big](\tau,-z,t)
\\[2mm]
& & \hspace{5mm}
\times \,\
e^{\pi i(K+3)z} \big[
e^{\frac{\pi in_2}{2}} \, \widetilde{\theta}_{n_2-(K+3),2(K+3)}
-
e^{-\frac{\pi in_2}{2}} \, \widetilde{\theta}_{-n_2-(K+3),2(K+3)}
\big](\tau,z,t)
\\[2mm]
&=&
- \, e^{\frac{\pi i}{2}(n_2-n_1)} \, 
\big[\widetilde{\theta}_{n_1+(K+3),2(K+3)}-e^{\pi in_1} \, 
\widetilde{\theta}_{-n_1+(K+3),2(K+3)}\big](\tau,z,t)
\\[2mm]
& & \hspace{13mm}
\times \,\ \big[
\widetilde{\theta}_{n_2-(K+3),2(K+3)}-e^{-\pi in_2} \, 
\widetilde{\theta}_{-n_2-(K+3),2(K+3)}\big](\tau,z,t)
\end{eqnarray*}}
proving Proposition \ref{prop:2025-720a}.
\end{proof}

\medskip

By \eqref{eqn:2025-721d}, the formulas in 
the above Proposition \ref{prop:2025-720a} are written 
in terms of $f_{j,2(K+3)}^{(\pm)}$'s as follows:

\medskip



\begin{prop} 
\label{paper:prop:2025-1105b}
Under the same assumption with Proposition \ref{prop:2025-720a},
the following formulas hold:

\begin{enumerate}
\item[{\rm 1)}]
\begin{subequations}
\begin{enumerate}
\item[{\rm (i)}] \,\ $4 \, \big[\overset{w}{R}{}^{(+)} \cdot \, 
\overset{w}{\rm ch}{}^{(+)}_{H(\Lambda_{n_1,n_2}^{[K]})}
\big](\tau,z\alpha_2,t)$
{\allowdisplaybreaks
\begin{eqnarray}
& & \hspace{-20mm}
= \,\ \Big\{
\big(f_{-n_1+(K+3),2(K+3)}^{(-)}+f_{n_1-(K+3),2(K+3)}^{(-)}\big)
+ 
\big(f_{-n_1+(K+3),2(K+3)}^{(+)}-f_{n_1-(K+3),2(K+3)}^{(+)}\big)\Big\}
\nonumber
\\[2mm]
& & \hspace{-18mm}
\times \, \Big\{
- \big(f_{-n_2+(K+3),2(K+3)}^{(-)}+f_{n_2-(K+3),2(K+3)}^{(-)}\big)
+ 
\big(f_{-n_2+(K+3),2(K+3)}^{(+)}-f_{n_2-(K+3),2(K+3)}^{(+)}\big)\Big\}
\nonumber
\\[0mm]
& &
\label{eqn:2025-724b1}
\end{eqnarray}}

\vspace{-7mm}
\item[{\rm (ii)}] \,\ $4 \, \big[\overset{w}{R}{}^{(+)} \cdot \, 
\overset{w}{\rm ch}{}^{(+)}_{H(
\Lambda_{2(K+3)-n_1,n_2}^{[K]})}\big](\tau,z\alpha_2,t)$
{\allowdisplaybreaks
\begin{eqnarray}
& & \hspace{-20mm}
= \,\ \Big\{
\big(f_{-n_1+(K+3),2(K+3)}^{(-)}+f_{n_1-(K+3),2(K+3)}^{(-)}\big)
- 
\big(f_{-n_1+(K+3),2(K+3)}^{(+)}-f_{n_1-(K+3),2(K+3)}^{(+)}\big)\Big\}
\nonumber
\\[2mm]
& & \hspace{-18mm}
\times \, \Big\{
- \big(f_{-n_2+(K+3),2(K+3)}^{(-)}+f_{n_2-(K+3),2(K+3)}^{(-)}\big)
+ 
\big(f_{-n_2+(K+3),2(K+3)}^{(+)}-f_{n_2-(K+3),2(K+3)}^{(+)}\big)\Big\}
\nonumber
\\[0mm]
& &
\label{eqn:2025-724b2}
\end{eqnarray}}

\vspace{-7mm}
\item[{\rm (iii)}] \,\ $4 \, \big[\overset{w}{R}{}^{(+)} \cdot \, 
\overset{w}{\rm ch}{}^{(+)}_{H(\Lambda_{n_1,2(K+3)-n_2}^{[K]})}
\big](\tau,z\alpha_2,t)$
{\allowdisplaybreaks
\begin{eqnarray}
& & \hspace{-20mm}
= \,\ - \Big\{
\big(f_{-n_1+(K+3),2(K+3)}^{(-)}+f_{n_1-(K+3),2(K+3)}^{(-)}\big)
+ 
\big(f_{-n_1+(K+3),2(K+3)}^{(+)}-f_{n_1-(K+3),2(K+3)}^{(+)}\big)\Big\}
\nonumber
\\[2mm]
& & \hspace{-18mm}
\times \, \Big\{
\big(f_{-n_2+(K+3),2(K+3)}^{(-)}+f_{n_2-(K+3),2(K+3)}^{(-)}\big)
+ 
\big(f_{-n_2+(K+3),2(K+3)}^{(+)}-f_{n_2-(K+3),2(K+3)}^{(+)}\big)\Big\}
\nonumber
\\[0mm]
& &
\label{eqn:2025-724b3}
\end{eqnarray}}

\vspace{-7mm}
\item[{\rm (iv)}] \,\ $4 \, \big[\overset{w}{R}{}^{(+)} \cdot \, 
\overset{w}{\rm ch}{}^{(+)}_{H(
\Lambda_{2(K+3)-n_1,2(K+3)-n_2}^{[K]})}\big](\tau,z\alpha_2,t)$
{\allowdisplaybreaks
\begin{eqnarray}
& & \hspace{-20mm}
= \,\ \Big\{
-\big(f_{-n_1+(K+3),2(K+3)}^{(-)}+f_{n_1-(K+3),2(K+3)}^{(-)}\big)
+ 
\big(f_{-n_1+(K+3),2(K+3)}^{(+)}-f_{n_1-(K+3),2(K+3)}^{(+)}\big)\Big\}
\nonumber
\\[2mm]
& & \hspace{-18mm}
\times \, \Big\{
\big(f_{-n_2+(K+3),2(K+3)}^{(-)}+f_{n_2-(K+3),2(K+3)}^{(-)}\big)
+ 
\big(f_{-n_2+(K+3),2(K+3)}^{(+)}-f_{n_2-(K+3),2(K+3)}^{(+)}\big)\Big\}
\nonumber
\\[0mm]
& &
\label{eqn:2025-724b4}
\end{eqnarray}}
\end{enumerate}
\end{subequations}

\vspace{-3mm}

\item[{\rm 2)}]
\begin{subequations}
\begin{enumerate}
\item[{\rm (i)}] \,\ $4 \, \big[\overset{w}{R}{}^{(-)} \cdot \, 
\overset{w}{\rm ch}{}^{(-)}_{H(\Lambda_{n_1,n_2}^{[K]})}
\big](\tau,z\alpha_2,t) 
\,\ = \,\ e^{\frac{\pi i}{2}(n_1-n_2)}$
{\allowdisplaybreaks
\begin{eqnarray}
& & \hspace{-20mm}
\times \, \Big\{
\big(f_{-n_1+(K+3),2(K+3)}^{(-)}-f_{n_1-(K+3),2(K+3)}^{(-)}\big)
+ 
\big(f_{-n_1+(K+3),2(K+3)}^{(+)}+f_{n_1-(K+3),2(K+3)}^{(+)}\big)\Big\}
\nonumber
\\[2mm]
& & \hspace{-20mm}
\times \, \Big\{
- 
\big(f_{-n_2+(K+3),2(K+3)}^{(-)}
+ 
f_{n_2-(K+3),2(K+3)}^{(-)}\big)
+ 
\big(f_{-n_2+(K+3),2(K+3)}^{(+)}
- 
f_{n_2-(K+3),2(K+3)}^{(+)}\big)\Big\}
\nonumber
\\[0mm]
& &
\label{eqn:2025-724c1}
\end{eqnarray}}

\vspace{-7mm}
\item[{\rm (ii)}] \,\ $4 \, \big[\overset{w}{R}{}^{(-)} \cdot \, 
\overset{w}{\rm ch}{}^{(-)}_{H(
\Lambda_{2(K+3)-n_1,n_2}^{[K]})}\big](\tau,z\alpha_2,t)
\,\ = \,\ e^{\frac{\pi i}{2}(n_1-n_2)}$
{\allowdisplaybreaks
\begin{eqnarray}
& & \hspace{-20mm}
\times \, \Big\{
\big(f_{-n_1+(K+3),2(K+3)}^{(-)} - f_{n_1-(K+3),2(K+3)}^{(-)}\big)
- 
\big(f_{-n_1+(K+3),2(K+3)}^{(+)}+f_{n_1-(K+3),2(K+3)}^{(+)}\big)\Big\}
\nonumber
\\[2mm]
& & \hspace{-20mm}
\times \, \Big\{
- 
\big(f_{-n_2+(K+3),2(K+3)}^{(-)}
+ 
f_{n_2-(K+3),2(K+3)}^{(-)}\big)
+ 
\big(f_{-n_2+(K+3),2(K+3)}^{(+)}
- 
f_{n_2-(K+3),2(K+3)}^{(+)}\big)\Big\}
\nonumber
\\[0mm]
& &
\label{eqn:2025-724c2}
\end{eqnarray}}

\vspace{-7mm}
\item[{\rm (iii)}] \,\ $4 \, \big[\overset{w}{R}{}^{(+)} \cdot \, 
\overset{w}{\rm ch}{}^{(+)}_{H(\Lambda_{n_1,2(K+3)-n_2}^{[K]})}
\big](\tau,z\alpha_2,t)
\,\ = \,\ - \,\ e^{\frac{\pi i}{2}(n_1-n_2)}$
{\allowdisplaybreaks
\begin{eqnarray}
& & \hspace{-20mm}
\times \, \Big\{
\big(f_{-n_1+(K+3),2(K+3)}^{(-)}-f_{n_1-(K+3),2(K+3)}^{(-)}\big)
+ 
\big(f_{-n_1+(K+3),2(K+3)}^{(+)}+f_{n_1-(K+3),2(K+3)}^{(+)}\big)\Big\}
\nonumber
\\[2mm]
& & \hspace{-20mm}
\times \, \Big\{
\big(f_{-n_2+(K+3),2(K+3)}^{(-)}+f_{n_2-(K+3),2(K+3)}^{(-)}\big)
+ 
\big(f_{-n_2+(K+3),2(K+3)}^{(+)}-f_{n_2-(K+3),2(m+3)}^{(+)}\big)\Big\}
\nonumber
\\[0mm]
& &
\label{eqn:2025-724c3}
\end{eqnarray}}

\vspace{-7mm}
\item[{\rm (iv)}] \,\ $4 \, \big[\overset{w}{R}{}^{(+)} \cdot \, 
\overset{w}{\rm ch}{}^{(+)}_{H(
\Lambda_{2(K+3)-n_1,2(K+3)-n_2}^{[K]})}\big](\tau,z\alpha_2,t)
\,\ = \,\ e^{\frac{\pi i}{2}(n_1-n_2)}$
{\allowdisplaybreaks
\begin{eqnarray}
& & \hspace{-20mm}
\times \, \Big\{
-\big(f_{-n_1+(K+3),2(K+3)}^{(-)}-f_{n_1-(K+3),2(K+3)}^{(-)}\big)
+ 
\big(f_{-n_1+(K+3),2(K+3)}^{(+)}+f_{n_1-(K+3),2(K+3)}^{(+)}\big)\Big\}
\nonumber
\\[2mm]
& & \hspace{-18mm}
\times \, \Big\{
\big(f_{-n_2+(K+3),2(K+3)}^{(-)}+f_{n_2-(K+3),2(K+3)}^{(-)}\big)
+ 
\big(f_{-n_2+(K+3),2(K+3)}^{(+)}-f_{n_2-(K+3),2(K+3)}^{(+)}\big)\Big\}
\nonumber
\\[0mm]
& &
\label{eqn:2025-724c4}
\end{eqnarray}}
\end{enumerate}
\end{subequations}
\end{enumerate}
\end{prop}

\medskip

\begin{note} \,\
\label{note:2025-1020b}
Under the same assumption with Proposition \ref{prop:2025-720a},
the following formula holds:
\begin{enumerate}
\item[{\rm 1)}] \, 
$\big[\overset{w}{R}{}^{(\pm)}\cdot \, 
\overset{w}{\rm ch}{}^{(\pm)}_{H(
\Lambda_{2(K+3)-n_2,2(K+3)-n_1}^{[K]})}\big]
(\tau,z\alpha_2,t)
\,\ = \,\ \pm \, 
\big[\overset{w}{R}{}^{(\pm)}\cdot \, 
\overset{w}{\rm ch}{}^{(\pm)}_{H(
\Lambda_{n_1,n_2}^{[K]})}\big](\tau,z\alpha_2,t)$

\item[{\rm 2)}] \, So we may consider the quantum Hamiltonian reduction
only in the case $n_1 \in \zzz_{\rm odd}$ and $n_2 \in \zzz_{\rm even}$. 
\end{enumerate}
\end{note}

\medskip

We now going to compute the modular transformation properties 
of the numerators of the quantum Hamiltonian reduction by 
using the following formulas. 

\medskip

\begin{note} 
\label{note:2025-724c}
For $K \in \zzz_{\geq -1, \, {\rm odd}}$ and $j \in \zzz$,
the following formulas hold:
\begin{subequations}
\begin{enumerate}
\item[{\rm 1)}] \quad $(f_{j,2(K+3)}^{(\pm)} \pm 
f_{-j,2(K+3)}^{(\pm)})|_{ST^2S} \, = \,\ 
\dfrac{-i \, \gamma_{2(K+3)}}{4(K+3)}$
\begin{equation} \hspace{-5mm}
\times 
\sum_{\substack{\\[0.5mm] 
-(K+3) \, < \, k \, \leq K+3
\\[1mm] k \, \equiv \, j \, {\rm mod} \, 2}} \hspace{-2mm}
e^{-\frac{\pi i}{8(K+3)}(j+k)^2} \, \big\{
1 \pm (-1)^{\frac{j+k}{2}} \, (-1)^{\frac{K+3}{2}}\big\} \,
(f_{k,2(K+3)}^{(\pm)} \, \pm \, f_{-k,2(K+3)}^{(\pm)})
\label{eqn:2025-724a1}
\end{equation}

\item[{\rm 2)}] \quad $(f_{j,2(K+3)}^{(\pm)} \mp 
f_{-j,2(K+3)}^{(\pm)})|_{ST^2S} \, = \,\ 
\dfrac{-i \, \gamma_{2(K+3)}}{4(K+3)}$
\begin{equation} \hspace{-5mm}
\times 
\sum_{\substack{\\[0.5mm] 
-(K+3) \, < \, k \, < K+3
\\[1mm] k \, \equiv \, j \, {\rm mod} \, 2}} \hspace{-2mm}
e^{-\frac{\pi i}{8(K+3)}(j+k)^2} \, \big\{
1 \pm (-1)^{\frac{j+k}{2}} \, (-1)^{\frac{K+3}{2}}\big\} \,
(f_{k,2(K+3)}^{(\pm)} \, \mp \, f_{-k,2(K+3)}^{(\pm)})
\label{eqn:2025-724a2}
\end{equation}

\item[{\rm 3)}] \quad In the case $j \in \zzz_{\rm odd}$, 
$k=K+3$ does not occur in the RHS of \eqref{eqn:2025-724a1}, 
so we have
{\allowdisplaybreaks
\begin{eqnarray}
& & \hspace{-20mm}
(f_{j,2(K+3)}^{(\pm)} \pm 
f_{-j,2(K+3)}^{(\pm)})|_{ST^2S} \, = \,\ 
\dfrac{-i \, \gamma_{2(K+3)}}{4(K+3)}
\nonumber
\\[2mm]
& & \hspace{-30mm}
\times 
\sum_{\substack{\\[0.5mm] 
-(K+3) \, < \, k \, < K+3
\\[1mm] k \, \equiv \, j \, {\rm mod} \, 2}} \hspace{-2mm}
e^{-\frac{\pi i}{8(K+3)}(j+k)^2} \, \big\{
1 \pm (-1)^{\frac{j+k}{2}} \, (-1)^{\frac{K+3}{2}}\big\} \,
(f_{k,2(K+3)}^{(\pm)} \, \pm \, f_{-k,2(K+3)}^{(\pm)})
\label{eqn:2025-724a3}
\end{eqnarray}}
\end{enumerate}
\end{subequations}
\end{note}

\begin{proof} These formulas are obtained immediately from 
Lemma \ref{lemma:2025-721d} by letting $m=2(K+3)$.
\end{proof}

\medskip

By easy calculation using formulas in the above 
Note \ref{note:2025-724c}, we obtain the following:

\medskip

\begin{note} \quad 
\label{note:2025-724d}
For $K \in \zzz_{\geq -1, \, {\rm odd}}$ and $j \in \zzz$ and 
$j \in \zzz$, the following formulas hold:

\begin{enumerate}
\item[{\rm 1)}]
\begin{enumerate}
\item[{\rm (i)}] \,\ $\big[
(f_{j,2(K+3)}^{(+)} - f_{-j,2(K+3)}^{(+)})
\, + \, 
(f_{j,2(K+3)}^{(-)} + f_{-j,2(K+3)}^{(-)})
\big]\big|_{ST^2S} $
{\allowdisplaybreaks
\begin{eqnarray*}
& & \hspace{-17mm}
= \,\ \frac{-i \, \gamma_{2(K+3)}}{4(K+3)}
\sum_{\substack{\\[0.5mm] 
-(K+3) \, < \, k \, < K+3
\\[1mm] k \, \equiv \, j \, {\rm mod} \, 2}} \hspace{-2mm}
e^{-\frac{\pi i}{8(K+3)}(j+k)^2} \, \Big\{
\\[2mm]
& & \hspace{-10mm}
\big[(f_{k,2(K+3)}^{(+)} \, - \, f_{-k,2(K+3)}^{(+)})
+
(f_{k,2(K+3)}^{(-)} \, + \, f_{-k,2(K+3)}^{(-)})\big]
\\[2mm]
& & \hspace{-15mm}
+ \,\ (-1)^{\frac{j+k}{2}} \, (-1)^{\frac{K+3}{2}} \, 
\big[(f_{k,2(K+3)}^{(+)} \, - \, f_{-k,2(K+3)}^{(+)})
-
(f_{k,2(K+3)}^{(-)} \, + \, f_{-k,2(K+3)}^{(-)})\big]
\Big\}
\end{eqnarray*}}

\vspace{-2mm}

\item[{\rm (ii)}] \,\ $\big[
(f_{j,2(m+3)}^{(+)} - f_{-j,2(K+3)}^{(+)})
\, - \, 
(f_{j,2(m+3)}^{(-)} + f_{-j,2(K+3)}^{(-)})
\big]\big|_{ST^2S} 
\, = \,\ 
\dfrac{-i \, \gamma_{2(K+3)}}{4(K+3)}$
{\allowdisplaybreaks
\begin{eqnarray*}
& & \hspace{-17mm}
= \,\ \frac{-i \, \gamma_{2(K+3)}}{4(K+3)}
\sum_{\substack{\\[0.5mm] 
-(K+3) \, < \, k \, < K+3
\\[1mm] k \, \equiv \, j \, {\rm mod} \, 2}} \hspace{-2mm}
e^{-\frac{\pi i}{8(K+3)}(j+k)^2} \, \Big\{
\\[2mm]
& & \hspace{-10mm}
\big[(f_{k,2(K+3)}^{(+)} \, - \, f_{-k,2(K+3)}^{(+)})
-
(f_{k,2(K+3)}^{(-)} \, + \, f_{-k,2(K+3)}^{(-)})\big]
\\[2mm]
& & \hspace{-15mm}
+ \,\ (-1)^{\frac{j+k}{2}} \, (-1)^{\frac{K+3}{2}} \, 
\big[(f_{k,2(K+3)}^{(+)} \, - \, f_{-k,2(K+3)}^{(+)})
+
(f_{k,2(K+3)}^{(-)} \, + \, f_{-k,2(K+3)}^{(-)})\big]
\Big\}
\end{eqnarray*}}
\end{enumerate}

\item[{\rm 2)}] \quad If $j \in \zzz_{\rm odd}$, then

\begin{enumerate}
\item[{\rm (i)}] \,\ $\big[
(f_{j,2(K+3)}^{(+)} + f_{-j,2(K+3)}^{(+)})
\, + \, 
(f_{j,2(K+3)}^{(-)} - f_{-j,2(K+3)}^{(-)})
\big]\big|_{ST^2S} $
{\allowdisplaybreaks
\begin{eqnarray*}
& & \hspace{-17mm}
= \,\ \frac{-i \, \gamma_{2(K+3)}}{4(K+3)}
\sum_{\substack{\\[0.5mm] 
-(K+3) \, < \, k \, < K+3
\\[1mm] k \, \equiv \, j \, {\rm mod} \, 2}} \hspace{-2mm}
e^{-\frac{\pi i}{8(K+3)}(j+k)^2} \, \Big\{
\\[2mm]
& & \hspace{-10mm}
\big[(f_{k,2(K+3)}^{(+)} \, + \, f_{-k,2(K+3)}^{(+)})
+
(f_{k,2(K+3)}^{(-)} \, - \, f_{-k,2(K+3)}^{(-)})\big]
\\[2mm]
& & \hspace{-15mm}
+ \,\ (-1)^{\frac{j+k}{2}} \, (-1)^{\frac{K+3}{2}} \, 
\big[(f_{k,2(K+3)}^{(+)} \, + \, f_{-k,2(K+3)}^{(+)})
-
(f_{k,2(K+3)}^{(-)} \, - \, f_{-k,2(K+3)}^{(-)})\big]
\Big\}
\end{eqnarray*}}

\vspace{-2mm}

\item[{\rm (ii)}] \,\ $\big[
(f_{j,2(K+3)}^{(+)} + f_{-j,2(K+3)}^{(+)})
\, - \, 
(f_{j,2(K+3)}^{(-)} - f_{-j,2(K+3)}^{(-)})
\big]\big|_{ST^2S} $
{\allowdisplaybreaks
\begin{eqnarray*}
& & \hspace{-17mm}
= \,\ \frac{-i \, \gamma_{2(K+3)}}{4(K+3)}
\sum_{\substack{\\[0.5mm] 
-(K+3) \, < \, k \, < K+3
\\[1mm] k \, \equiv \, j \, {\rm mod} \, 2}} \hspace{-2mm}
e^{-\frac{\pi i}{8(K+3)}(j+k)^2} \, \Big\{
\\[2mm]
& & \hspace{-10mm}
\big[(f_{k,2(K+3)}^{(+)} \, + \, f_{-k,2(K+3)}^{(+)})
-
(f_{k,2(K+3)}^{(-)} \, - \, f_{-k,2(K+3)}^{(-)})\big]
\\[2mm]
& & \hspace{-15mm}
+ \,\ (-1)^{\frac{j+k}{2}} \, (-1)^{\frac{K+3}{2}} \, 
\big[(f_{k,2(K+3)}^{(+)} \, + \, f_{-k,2(K+3)}^{(+)})
+
(f_{k,2(K+3)}^{(-)} \, - \, f_{-k,2(K+3)}^{(-)})\big]
\Big\}
\end{eqnarray*}}
\end{enumerate}
\end{enumerate}
\end{note}

\medskip

Letting $j \rightarrow -j+(K+3)$ and $k \rightarrow -k+(K+3)$
in the above Note \ref{note:2025-724d} gives the following formulas.

\medskip

\begin{note} 
\label{note:2025-724e}
For $K \in \zzz_{\geq -1, \, {\rm odd}}$ and $j \in \zzz$, 
the following formulas hold:

\begin{enumerate}
\item[{\rm 1)}]
\begin{enumerate}
\item[{\rm (i)}] $\big[
(f_{-j+(K+3),2(K+3)}^{(+)} - f_{j-(K+3),2(K+3)}^{(+)})
\, + \, 
(f_{-j+(K+3),2(K+3)}^{(-)} + f_{j-(K+3),2(K+3)}^{(-)})
\big]\big|_{ST^2S} $
{\allowdisplaybreaks
\begin{eqnarray*}
& & \hspace{-17mm}
= \,\ \frac{-i \, \gamma_{2(K+3)}}{4(K+3)}
\sum_{\substack{\\[0.5mm] 
0 \, < \, k \, < \, 2(K+3)
\\[1mm] k \, \equiv \, j \, {\rm mod} \, 2}} \hspace{-2mm}
e^{-\frac{\pi i}{8(K+3)}(j+k)^2} \, \Big\{
(-1)^{\frac{j+k}{2}} \, (-1)^{\frac{K+3}{2}} \, 
\\[2mm]
& & \hspace{-10mm}
\times \, 
\big[(f_{-k+(K+3),2(K+3)}^{(+)} \, - \, f_{k-(K+3),2(K+3)}^{(+)})
+
(f_{-k+(K+3),2(K+3)}^{(-)} \, + \, f_{k-(K+3),2(K+3)}^{(-)})\big]
\\[2mm]
& & \hspace{-15mm}
+ \,\ 
\big[(f_{-k+(K+3),2(K+3)}^{(+)} \, - \, f_{k-(K+3),2(K+3)}^{(+)})
-
(f_{-k+(K+3),2(K+3)}^{(-)} \, + \, f_{k-(K+3),2(K+3)}^{(-)})\big]
\Big\}
\end{eqnarray*}}

\vspace{-2mm}

\item[{\rm (ii)}] \,\ $\big[
(f_{-j+(K+3),2(K+3)}^{(+)} - f_{j-(K+3),2(K+3)}^{(+)})
\, - \, 
(f_{-j+(K+3),2(K+3)}^{(-)} + f_{j-(K+3),2(K+3)}^{(-)})
\big]\big|_{ST^2S} $
{\allowdisplaybreaks
\begin{eqnarray*}
& & \hspace{-17mm}
= \,\ \frac{-i \, \gamma_{2(K+3)}}{4(K+3)}
\sum_{\substack{\\[0.5mm] 
0 \, < \, k \, < \, 2(K+3)
\\[1mm] k \, \equiv \, j \, {\rm mod} \, 2}} \hspace{-2mm}
e^{-\frac{\pi i}{8(K+3)}(j+k)^2} \, \Big\{ \, 
(-1)^{\frac{j+k}{2}} \, (-1)^{\frac{K+3}{2}} \,
\\[2mm]
& & \hspace{-10mm}
\times \, 
\big[(f_{-k+(K+3),2(K+3)}^{(+)} \, - \, f_{k-(K+3),2(K+3)}^{(+)})
-
(f_{-k+(K+3),2(K+3)}^{(-)} \, + \, f_{k-(K+3),2(K+3)}^{(-)})\big]
\\[2mm]
& & \hspace{-15mm}
+ \,\  
\big[(f_{-k+(K+3),2(K+3)}^{(+)} \, - \, f_{k-(K+3),2(K+3)}^{(+)})
+
(f_{-k+(K+3),2(K+3)}^{(-)} \, + \, f_{k-(K+3),2(K+3)}^{(-)})\big]
\Big\}
\end{eqnarray*}}
\end{enumerate}

\item[{\rm 2)}] \,\ If $j \in \zzz_{\rm odd}$, then

\begin{enumerate}
\item[{\rm (i)}] \,\ $\big[
(f_{-j+(K+3),2(K+3)}^{(+)} + f_{j-(K+3),2(K+3)}^{(+)})
\, + \, 
(f_{-j+(K+3),2(K+3)}^{(-)} - f_{j-(K+3),2(K+3)}^{(-)})
\big]\big|_{ST^2S} $
{\allowdisplaybreaks
\begin{eqnarray*}
& & \hspace{-17mm}
= \,\ \frac{-i \, \gamma_{2(K+3)}}{4(K+3)}
\sum_{\substack{\\[0.5mm] 
0 \, < \, k \, < \, 2(K+3)
\\[1mm] k \, \equiv \, j \, {\rm mod} \, 2}} \hspace{-2mm}
e^{-\frac{\pi i}{8(K+3)}(j+k)^2} \, \Big\{ \, 
(-1)^{\frac{j+k}{2}} \, (-1)^{\frac{K+3}{2}} \, 
\\[2mm]
& & \hspace{-10mm}
\times \, 
\big[(f_{-k+(K+3),2(K+3)}^{(+)} \, + \, f_{k-(K+3),2(K+3)}^{(+)})
+
(f_{-k+(K+3),2(K+3)}^{(-)} \, - \, f_{k-(K+3),2(K+3)}^{(-)})\big]
\\[2mm]
& & \hspace{-15mm}
+ \,\ 
\big[(f_{-k+(K+3),2(K+3)}^{(+)} \, + \, f_{k-(K+3),2(K+3)}^{(+)})
-
(f_{-k+(K+3),2(K+3)}^{(-)} \, - \, f_{k-(K+3),2(K+3)}^{(-)})\big]
\Big\}
\end{eqnarray*}}

\vspace{-2mm}

\item[{\rm (ii)}] \,\ $\big[
(f_{-j+(K+3),2(K+3)}^{(+)} + f_{j-(K+3),2(K+3)}^{(+)})
\, - \, 
(f_{-j+(K+3),2(K+3)}^{(-)} - f_{j-(K+3),2(K+3)}^{(-)})
\big]\big|_{ST^2S} $
{\allowdisplaybreaks
\begin{eqnarray*}
& & \hspace{-17mm}
= \,\ \frac{-i \, \gamma_{2(K+3)}}{4(K+3)}
\sum_{\substack{\\[0.5mm] 
0 \, < \, k \, < \, 2(K+3)
\\[1mm] k \, \equiv \, j \, {\rm mod} \, 2}} \hspace{-2mm}
e^{-\frac{\pi i}{8(K+3)}(j+k)^2} \, \Big\{ \, 
(-1)^{\frac{j+k}{2}} \, (-1)^{\frac{K+3}{2}} \, 
\\[2mm]
& & \hspace{-10mm}
\times \, 
\big[(f_{-k+(K+3),2(K+3)}^{(+)} \, + \, f_{k-(K+3),2(K+3)}^{(+)})
-
(f_{-k+(K+3),2(K+3)}^{(-)} \, - \, f_{k-(K+3),2(K+3)}^{(-)})\big]
\\[2mm]
& & \hspace{-15mm}
+ \,\ 
\big[(f_{-k+(K+3),2(K+3)}^{(+)} \, + \, f_{k-(K+3),2(K+3)}^{(+)})
+
(f_{-k+(K+3),2(K+3)}^{(-)} \, - \, f_{k-(K+3),2(K+3)}^{(-)})\big]
\Big\}
\end{eqnarray*}}
\end{enumerate}
\end{enumerate}
\end{note}

\medskip

\begin{prop} 
\label{prop:2025-725a}
For $\Lambda^{[K]}_{n_1,n_2} \in P^+_K(S^{\vee}_{(u=1)})$ such that 
$n_1 \in \zzz_{\rm odd}$ and $n_2 \in \zzz_{\rm even}$, 
the $ST^2S$- and $T$-transformation of 
$\overset{w}{R}{}^{(\pm)} \cdot \, 
\overset{w}{\rm ch}{}^{(\pm)}_{H(\Lambda_{n_1,n_2}^{[K]})}$
are given by the following formulas.

\begin{enumerate}
\item[{\rm 1)}]
\begin{enumerate}
\item[{\rm (i)}] \,\ $\Big(\overset{w}{R}{}^{(+)} \cdot \, 
\overset{w}{\rm ch}{}^{(+)}_{H(\Lambda_{n_1,n_2}^{[K]})}
\Big)\Big|_{ST^2S}
= \,\
- \, \Big[\dfrac{\gamma_{2(K+3)}}{4(K+3)}\Big]^2
\sum\limits_{\substack{0 \, < \, k_1, \, k_2 \, < \, 2(K+3) \\[1mm]
k_1 \, \in \, \zzz_{\rm odd} \\[1mm]
k_2 \, \in \, \zzz_{\rm even} }} \hspace{-5mm}
(-1)^{\frac12(n_1+n_2+k_1+k_2)} $
$$
\times \, \Big(
e^{-\frac{\pi i(n_1+k_1)^2}{8(K+3)}} \, 
+
e^{-\frac{\pi i(n_1-k_1)^2}{8(K+3)}}\Big) \, \Big(
e^{-\frac{\pi i(n_2+k_2)^2}{8(K+3)}} \, 
-
e^{-\frac{\pi i(n_2-k_2)^2}{8(K+3)}}\Big) \, 
\overset{w}{R}{}^{(+)} \cdot \, 
\overset{w}{\rm ch}{}^{(+)}_{H(\Lambda_{k_1,k_2}^{[K]})}
$$

\item[{\rm (ii)}] \,\ $\Big(\overset{w}{R}{}^{(-)} \cdot \, 
\overset{w}{\rm ch}{}^{(-)}_{H(\Lambda_{n_1,n_2}^{[K]})}
\Big)\Big|_{ST^2S}
= \,\
\Big[\dfrac{\gamma_{2(K+3)}}{4(K+3)}\Big]^2
\sum\limits_{\substack{0 \, < \, k_1, \, k_2 \, < \, 2(K+3) \\[1mm]
k_1 \, \in \, \zzz_{\rm odd} \\[1mm]
k_2 \, \in \, \zzz_{\rm even} }} $
$$
\times \, \Big(
e^{-\frac{\pi i(n_1+k_1)^2}{8(K+3)}} \, 
+
e^{-\frac{\pi i(n_1-k_1)^2}{8(K+3)}}\Big) \, \Big(
e^{-\frac{\pi i(n_2+k_2)^2}{8(K+3)}} \, 
-
e^{-\frac{\pi i(n_2-k_2)^2}{8(K+3)}}\Big) \, 
\overset{w}{R}{}^{(-)} \cdot \, 
\overset{w}{\rm ch}{}^{(-)}_{H(\Lambda_{k_1,k_2}^{[K]})}
$$
\end{enumerate}

\item[{\rm 2)}]
\begin{enumerate}
\item[{\rm (i)}] \,\ $\Big(\overset{w}{R}{}^{(+)} \cdot \, 
\overset{w}{\rm ch}{}^{(+)}_{H(\Lambda_{n_1,n_2}^{[K]})}
\Big)\Big|_{T}
= \,\
- \, (-1)^{\frac{K+3}{2}} \, 
e^{\frac{\pi i(n_1^2+n_2^2)}{4(K+3)}} \,\ 
\overset{w}{R}{}^{(-)} \cdot \, 
\overset{w}{\rm ch}{}^{(-)}_{H(\Lambda_{n_1,n_2}^{[K]})}$

\vspace{1mm}

\item[{\rm (ii)}] \,\ $\Big(\overset{w}{R}{}^{(-)} \cdot \, 
\overset{w}{\rm ch}{}^{(-)}_{H(\Lambda_{n_1,n_2}^{[K]})}
\Big)\Big|_{T}
= \,\
(-1)^{\frac{K+3}{2}} \, 
e^{\frac{\pi i(n_1^2+n_2^2)}{4(K+3)}} \,\ 
\overset{w}{R}{}^{(+)} \cdot \, 
\overset{w}{\rm ch}{}^{(+)}_{H(\Lambda_{n_1,n_2}^{[K]})}$
\end{enumerate}
\end{enumerate}
\end{prop}

\begin{proof} 1) (i) \,\ By Proposition \ref{paper:prop:2025-1105b} 
and Note \ref{note:2025-724e}, we have 
{\allowdisplaybreaks
\begin{eqnarray*}
& & \hspace{-2mm} 
4 \, \Big(\overset{w}{R}{}^{(+)} \cdot \, 
\overset{w}{\rm ch}{}^{(+)}_{H(\Lambda_{n_1,n_2}^{[K]})}
\Big)\Big|_{ST^2S}
\\[2mm]
& & \hspace{-7mm}
= \Big\{
\big(f_{-n_1+(K+3),2(K+3)}^{(+)}-f_{n_1-(K+3),2(K+3)}^{(+)}\big)
+ 
\big(f_{-n_1+(K+3),2(K+3)}^{(-)}+f_{n_1-(K+3),2(K+3)}^{(-)}\big)\Big\}
\Big|_{ST^2S}
\\[2mm]
& & \hspace{-5mm}
\times \Big\{ 
\big(f_{-n_2+(K+3),2(K+3)}^{(+)}-f_{n_2-(K+3),2(K+3)}^{(+)}\big)
- 
\big(f_{-n_2+(K+3),2(K+3)}^{(-)}+f_{n_2-(K+3),2(K+3)}^{(-)}\big)\Big\}
\Big|_{ST^2S}
\\[3mm]
& & \hspace{-7mm}
= \,\ \Big[\frac{-i\gamma_{2(K+3)}}{4(K+3)}\Big]^2
\sum_{\substack{0 \, < \, k_1, \, k_2 \, < \, 2(K+3) \\[1mm]
k_1 \, \in \, \zzz_{\rm odd} \\[1mm]
k_2 \, \in \, \zzz_{\rm even} }}
e^{-\frac{\pi i}{8(K+3)}[(n_1+k_1)^2+(n_2+k_2)^2]}
\\[2mm]
& & \hspace{-3mm}
\times \Big\{
(-1)^{\frac{n_1+k_1}{2}}(-1)^{\frac{K+3}{2}}
\\[2mm]
& & \hspace{-0mm}
\times \big[
(f^{(+)}_{-k_1+(K+3), 2(K+3)}-f^{(+)}_{k_1-(K+3), 2(K+3)})
+
(f^{(-)}_{-k_1+(K+3), 2(K+3)}+f^{(-)}_{k_1-(K+3), 2(K+3)})\big]
\\[2mm]
& & \hspace{-0mm}
+ \big[
(f^{(+)}_{-k_1+(K+3), 2(K+3)}-f^{(+)}_{k_1-(K+3), 2(K+3)})
-
(f^{(-)}_{-k_1+(K+3), 2(K+3)}+f^{(-)}_{k_1-(K+3), 2(K+3)})\big]
\Big\}
\\[2mm] %
& & \hspace{-3mm}
\times \Big\{
(-1)^{\frac{n_2+k_2}{2}}(-1)^{\frac{K+3}{2}}
\\[2mm]
& & \hspace{-0mm}
\times \big[
(f^{(+)}_{-k_2+(K+3), 2(K+3)}-f^{(+)}_{k_2-(K+3), 2(K+3)})
-
(f^{(-)}_{-k_2+(K+3), 2(K+3)}+f^{(-)}_{k_2-(K+3), 2(K+3)})\big]
\\[2mm]
& & \hspace{-0mm}
+ \big[
(f^{(+)}_{-k_2+(K+3), 2(K+3)}-f^{(+)}_{k_2-(K+3), 2(K+3)})
+
(f^{(-)}_{-k_2+(K+3), 2(K+3)}+f^{(-)}_{k_2-(K+3), 2(K+3)})\big]
\Big\}
\\[2mm] 
& & \hspace{-7mm}
= \,\ \Big[\frac{-i\gamma_{2(K+3)}}{4(K+3)}\Big]^2
\sum_{\substack{0 \, < \, k_1, \, k_2 \, < \, 2(K+3) \\[1mm]
k_1 \, \in \, \zzz_{\rm odd} \\[1mm]
k_2 \, \in \, \zzz_{\rm even} }}
e^{-\frac{\pi i}{8(K+3)}[(n_1+k_1)^2+(n_2+k_2)^2]} \, 
\Big\{ (A_1)+(A_2)+(A_3)+(A_4) \Big\}
\end{eqnarray*}}
where 
{\allowdisplaybreaks
\begin{eqnarray*}
& & \hspace{-10mm}
(A_1) \, := \,\ (-1)^{\frac{n_1+k_1}{2}}(-1)^{\frac{K+3}{2}}
\cdot 
(-1)^{\frac{n_2+k_2}{2}}(-1)^{\frac{K+3}{2}}
\\[2mm]
& & \hspace{-0mm}
\times \big[
(f^{(+)}_{-k_1+(K+3), 2(K+3)}-f^{(+)}_{k_1-(K+3), 2(K+3)})
+
(f^{(-)}_{-k_1+(K+3), 2(K+3)}+f^{(-)}_{k_1-(K+3), 2(K+3)})\big]
\\[2mm]
& & \hspace{-0mm}
\times \big[
(f^{(+)}_{-k_2+(K+3), 2(K+3)}-f^{(+)}_{k_2-(K+3), 2(K+3)})
-
(f^{(-)}_{-k_2+(K+3), 2(K+3)}+f^{(-)}_{k_2-(K+3), 2(K+3)})\big]
\\[2mm]
&=&
(-1)^{\frac12(n_1+n_2+k_1+k_2)} 
\, \times \, 4 \, 
\overset{w}{R}{}^{(+)} \cdot \, 
\overset{w}{\rm ch}{}^{(+)}_{H(\Lambda_{k_1,k_2}^{[K]})}
\\[2mm]
& & \hspace{-10mm}
(A_2) \, := \,\ (-1)^{\frac{n_1+k_1}{2}}(-1)^{\frac{K+3}{2}}
\\[2mm]
& & \hspace{-0mm}
\times \big[
(f^{(+)}_{-k_1+(K+3), 2(K+3)}-f^{(+)}_{k_1-(K+3), 2(K+3)})
+
(f^{(-)}_{-k_1+(K+3), 2(K+3)}+f^{(-)}_{k_1-(K+3), 2(K+3)})\big]
\\[2mm]
& & \hspace{-0mm}
\times \big[
(f^{(+)}_{-k_2+(K+3), 2(K+3)}-f^{(+)}_{k_2-(K+3), 2(K+3)})
+
(f^{(-)}_{-k_2+(K+3), 2(K+3)}+f^{(-)}_{k_2-(K+3), 2(K+3)})\big]
\\[2mm]
&=&
(-1)^{\frac{n_1+k_1}{2}}(-1)^{\frac{K+3}{2}}
\, \times \, (-4) \, 
\overset{w}{R}{}^{(+)} \cdot \, 
\overset{w}{\rm ch}{}^{(+)}_{H(\Lambda_{k_1,2(K+3)-k_2}^{[K]})}
\\[2mm]
& & \hspace{-10mm}
(A_3) \, := \,\ (-1)^{\frac{n_2+k_2}{2}}(-1)^{\frac{K+3}{2}}
\\[2mm]
& & \hspace{-0mm}
\times \big[
(f^{(+)}_{-k_1+(K+3), 2(K+3)}-f^{(+)}_{k_1-(K+3), 2(K+3)})
-
(f^{(-)}_{-k_1+(K+3), 2(K+3)}+f^{(-)}_{k_1-(K+3), 2(K+3)})\big]
\\[2mm]
& & \hspace{-0mm}
\times \big[
(f^{(+)}_{-k_2+(K+3), 2(K+3)}-f^{(+)}_{k_2-(K+3), 2(K+3)})
-
(f^{(-)}_{-k_2+(K+3), 2(K+3)}+f^{(-)}_{k_2-(K+3), 2(K+3)})\big]
\\[2mm]
&=&
(-1)^{\frac{n_2+k_2}{2}}(-1)^{\frac{K+3}{2}}
\, \times \, (-4) \, 
\overset{w}{R}{}^{(+)} \cdot \, 
\overset{w}{\rm ch}{}^{(+)}_{H(\Lambda_{2(K+3)-k_1,k_2}^{[K]})}
\\[2mm]
& & \hspace{-10mm}
(A_4) \, := \, \big[
(f^{(+)}_{-k_1+(K+3), 2(K+3)}-f^{(+)}_{k_1-(K+3), 2(K+3)})
-
(f^{(-)}_{-k_1+(K+3), 2(K+3)}+f^{(-)}_{k_1-(K+3), 2(K+3)})\big]
\\[2mm]
& & \hspace{-0mm}
\times \big[
(f^{(+)}_{-k_2+(K+3), 2(K+3)}-f^{(+)}_{k_2-(K+3), 2(K+3)})
+
(f^{(-)}_{-k_2+(K+3), 2(K+3)}+f^{(-)}_{k_2-(K+3), 2(K+3)})\big]
\\[2mm]
&=&
4 \,\ \overset{w}{R}{}^{(+)} \cdot \, 
\overset{w}{\rm ch}{}^{(+)}_{H(\Lambda_{2(K+3)-k_1,2(K+3)-k_2}^{[K]})}
\end{eqnarray*}}
Putting 
$$
(\widetilde{A}_j) \, := \,\ \frac14 \, 
\sum_{\substack{0 \, < \, k_1, \, k_2 \, < \, 2(K+3) \\[1mm]
k_1 \, \in \, \zzz_{\rm odd} \\[1mm]
k_2 \, \in \, \zzz_{\rm even} }}
e^{-\frac{\pi i}{8(K+3)}[(n_1+k_1)^2+(n_2+k_2)^2]} \, 
\times \, (A_j)
$$
we have 
$$
\Big(\overset{w}{R}{}^{(+)} \cdot \, 
\overset{w}{\rm ch}{}^{(+)}_{H(\Lambda_{n_1,n_2}^{[K]})}
\Big)\Big|_{ST^2S}
\, = \,\ \Big[\frac{-i\gamma_{2(K+3)}}{4(K+3)}\Big]^2
\, \times \, 
\Big\{(\widetilde{A}_1)+(\widetilde{A}_2)
+(\widetilde{A}_3)+(\widetilde{A}_4)\Big\}
$$
Below we are going to compute $(\widetilde{A}_1) \sim (\widetilde{A}_4)$
by putting $k'_j =2(K+3)-k_j$ \,\ $(j=1,2)$ at some suitable places.
By this putting, one has
$$
\frac{(n_j+k_j)^2}{8(K+3)} \,\ = \,\ \frac{(n_j+2(K+3)-k'_j)^2}{8(K+3)}
\, = \,\
\frac{(n_j-k'_j)^2}{8(K+3)}
-\frac{n_j-k'_j}{2}+\frac{K+3}{2}
$$
and so 
\begin{equation}
\left.
\begin{array}{l}
k_j =2(K+3)-k'_j \\[0mm]
n_j-k_j \, \in \, 2 \zzz
\end{array}\right\}
\,\ \Longrightarrow \quad 
e^{-\frac{\pi i(n_j+k_j)^2}{8(K+3)}}
\,\ = \,\ 
e^{-\frac{\pi i(n_j-k'_j)^2}{8(K+3)}} \, (-1)^{\frac{n_j-k'_j}{2}+\frac{K+3}{2}}
\label{eqn:2025-726a}
\end{equation}
Then $(\widetilde{A}_1) \sim (\widetilde{A}_4)$ become as follows:
{\allowdisplaybreaks
\begin{eqnarray*}
& & \hspace{-5mm}
(\widetilde{A}_1) \, = 
\sum_{\substack{0 \, < \, k_1, \, k_2 \, < \, 2(K+3) \\[1mm]
k_1 \, \in \, \zzz_{\rm odd} \\[1mm]
k_2 \, \in \, \zzz{\rm even} }} \hspace{-5mm}
e^{-\frac{\pi i(n_1+k_1)^2}{8(K+3)}} \, 
e^{-\frac{\pi i(n_2+k_2)^2}{8(K+3)}} \, 
(-1)^{\frac12(n_1+n_2+k_1+k_2)} 
\, \times \, 
\overset{w}{R}{}^{(+)} \cdot \, 
\overset{w}{\rm ch}{}^{(+)}_{H(\Lambda_{k_1,k_2}^{[K]})}
\\[2mm]
& & \hspace{-5mm}
(\widetilde{A}_2) \, = \,\ - \hspace{-5mm}
\sum_{\substack{0 \, < \, k_1, \, k_2 \, < \, 2(K+3) \\[1mm]
k_1 \, \in \, \zzz_{\rm odd} \\[1mm]
k_2 \, \in \, \zzz_{\rm even} }} \hspace{-5mm}
e^{-\frac{\pi i(n_1+k_1)^2}{8(K+3)}} \, 
e^{-\frac{\pi i(n_2+k_2)^2}{8(K+3)}} \, 
(-1)^{\frac{n_1+k_1}{2}}(-1)^{\frac{m+3}{2}} 
\, \times \, 
\overset{w}{R}{}^{(+)} \cdot \, 
\overset{w}{\rm ch}{}^{(+)}_{H(\Lambda_{k_1,2(K+3)-k_2}^{[K]})}
\\[2mm]
& &
= \,\ - \hspace{-5mm}
\sum_{\substack{0 \, < \, k_1, \, k'_2 \, < \, 2(K+3) \\[1mm]
k_1 \, \in \, \zzz_{\rm odd} \\[1mm]
k'_2 \, \in \, \zzz_{\rm even} }} \hspace{-5mm}
e^{-\frac{\pi i(n_1+k_1)^2}{8(K+3)}} \, 
e^{-\frac{(n_2-k'_2)^2}{8(K+3)}}
(-1)^{\frac{n_2-k'_2}{2}}
(-1)^{\frac{K+3}{2}}
(-1)^{\frac{n_1+k_1}{2}}(-1)^{\frac{K+3}{2}} 
\\[-10mm]
& & \hspace{50mm}
\times \,\ 
\overset{w}{R}{}^{(+)} \cdot \, 
\overset{w}{\rm ch}{}^{(+)}_{H(\Lambda_{k_1,k'_2}^{[K]})}
\\[2mm]
& &
= \,\ - \hspace{-5mm}
\sum_{\substack{0 \, < \, k_1, \, k'_2 \, < \, 2(K+3) \\[1mm]
k_1 \, \in \, \zzz_{\rm odd} \\[1mm]
k'_2 \, \in \, \zzz_{\rm even} }} \hspace{-5mm}
e^{-\frac{\pi i(n_1+k_1)^2}{8(K+3)}} \, 
e^{-\frac{(n_2-k'_2)^2}{8(K+3)}} \, 
(-1)^{\frac{n_1+k_1}{2}}
(-1)^{\frac{n_2+k'_2}{2}} 
\, \times \, 
\overset{w}{R}{}^{(+)} \cdot \, 
\overset{w}{\rm ch}{}^{(+)}_{H(\Lambda_{k_1,k'_2}^{[K]})}
\\[2mm]
& & \hspace{-5mm}
(\widetilde{A}_3) \, = \,\ - \hspace{-5mm}
\sum_{\substack{0 \, < \, k_1, \, k_2 \, < \, 2(K+3) \\[1mm]
k_1 \, \in \, \zzz_{\rm odd} \\[1mm]
k_2 \, \in \, \zzz_{\rm even} }} \hspace{-5mm}
e^{-\frac{\pi i(n_1+k_1)^2}{8(K+3)}} \, 
e^{-\frac{\pi i(n_2+k_2)^2}{8(K+3)}} \, 
(-1)^{\frac{n_2+k_2}{2}}(-1)^{\frac{K+3}{2}} 
\, \times \, 
\overset{w}{R}{}^{(+)} \cdot \, 
\overset{w}{\rm ch}{}^{(+)}_{H(\Lambda_{2(m+3)-k_1,k_2}^{[m]})}
\\[2mm]
& & 
= \,\ - \hspace{-5mm}
\sum_{\substack{0 \, < \, k'_1, \, k_2 \, < \, 2(K+3) \\[1mm]
k'_1 \, \in \, \zzz_{\rm odd} \\[1mm]
k_2 \, \in \, \zzz_{\rm even} }} \hspace{-5mm}
e^{-\frac{(n_1-k'_1)^2}{8(K+3)}} \, 
(-1)^{\frac{n_1-k'_1}{2}}
(-1)^{\frac{K+3}{2}}
e^{-\frac{\pi i(n_2+k_2)^2}{8(K+3)}} \, 
(-1)^{\frac{n_2+k_2}{2}}(-1)^{\frac{K+3}{2}} 
\\[-10mm]
& & \hspace{50mm}
\times \,\ 
\overset{w}{R}{}^{(+)} \cdot \, 
\overset{w}{\rm ch}{}^{(+)}_{H(\Lambda_{k'_1,k_2}^{[K]})}
\\[2mm]
& &
= \,\ 
\sum_{\substack{0 \, < \, k'_1, \, k_2 \, < \, 2(K+3) \\[1mm]
k'_1 \, \in \, \zzz_{\rm odd} \\[1mm]
k_2 \, \in \, \zzz_{\rm even} }} \hspace{-5mm}
e^{-\frac{\pi i(n_1-k'_1)^2}{8(K+3)}} \, 
e^{-\frac{(n_2+k_2)^2}{8(K+3)}} \, 
(-1)^{\frac{n_1+k'_1}{2}}
(-1)^{\frac{n_2+k_2}{2}} 
\, \times \, 
\overset{w}{R}{}^{(+)} \cdot \, 
\overset{w}{\rm ch}{}^{(+)}_{H(\Lambda_{k'_1,k_2}^{[K]})}
\\[2mm]
& & \hspace{-5mm}
(\widetilde{A}_4) \, =
\sum_{\substack{0 \, < \, k_1, \, k_2 \, < \, 2(K+3) \\[1mm]
k_1 \, \in \, \zzz_{\rm odd} \\[1mm]
k_2 \, \in \, \zzz_{\rm even} }} \hspace{-5mm}
e^{-\frac{\pi i(n_1+k_1)^2}{8(K+3)}} \, 
e^{-\frac{\pi i(n_2+k_2)^2}{8(K+3)}} 
\, \times \, 
\overset{w}{R}{}^{(+)} \cdot \, 
\overset{w}{\rm ch}{}^{(+)}_{H(\Lambda_{2(K+3)-k_1,2(K+3)-k_2}^{[K]})}
\\[2mm]
& & 
= \,\ 
\sum_{\substack{0 \, < \, k'_1, \, k_2 \, < \, 2(K+3) \\[1mm]
k'_1 \, \in \, \zzz_{\rm odd} \\[1mm]
k'_2 \, \in \, \zzz_{\rm even} }} \hspace{-5mm}
e^{-\frac{(n_1-k'_1)^2}{8(K+3)}} \, 
(-1)^{\frac{n_1-k'_1}{2}}
(-1)^{\frac{K+3}{2}}
e^{-\frac{(n_2-k'_2)^2}{8(K+3)}} \, 
(-1)^{\frac{n_2-k'_2}{2}}
(-1)^{\frac{K+3}{2}}
\\[-10mm]
& & \hspace{50mm}
\times \,\ 
\overset{w}{R}{}^{(+)} \cdot \, 
\overset{w}{\rm ch}{}^{(+)}_{H(\Lambda_{k'_1,k'_2}^{[K]})}
\\[2mm]
& &
= \,\ - \hspace{-5mm}
\sum_{\substack{0 \, < \, k'_1, \, k'_2 \, < \, 2(K+3) \\[1mm]
k'_1 \, \in \, \zzz_{\rm odd} \\[1mm]
k'_2 \, \in \, \zzz_{\rm even} }} \hspace{-5mm}
e^{-\frac{\pi i(n_1-k'_1)^2}{8(K+3)}} \, 
e^{-\frac{(n_2-k'_2)^2}{8(K+3)}} \, 
(-1)^{\frac{n_1+k'_1}{2}}
(-1)^{\frac{n_2+k'_2}{2}} 
\, \times \, 
\overset{w}{R}{}^{(+)} \cdot \, 
\overset{w}{\rm ch}{}^{(+)}_{H(\Lambda_{k'_1,k'_2}^{[K]})}
\end{eqnarray*}}
Thus we have 
{\allowdisplaybreaks
\begin{eqnarray*}
& & \hspace{-10mm}
\Big(\overset{w}{R}{}^{(+)} \cdot \, 
\overset{w}{\rm ch}{}^{(+)}_{H(\Lambda_{n_1,n_2}^{[K]})}
\Big)\Big|_{ST^2S}
\, = \,\ \Big[\frac{-i\gamma_{2(m+3)}}{4(m+3)}\Big]^2
\, \times \, 
\Big\{(\widetilde{A}_1)+(\widetilde{A}_2)
+(\widetilde{A}_3)+(\widetilde{A}_4)\Big\}
\\[2mm]
& & \hspace{-8mm}
= \,\
\Big[\frac{-i\gamma_{2(K+3)}}{4(K+3)}\Big]^2
\sum_{\substack{0 \, < \, k_1, \, k_2 \, < \, 2(K+3) \\[1mm]
k_1 \, \in \, \zzz_{\rm odd} \\[1mm]
k_2 \, \in \, \zzz_{\rm even} }} \hspace{-5mm}
(-1)^{\frac12(n_1+n_2+k_1+k_2)}
\\[2mm]
& & \hspace{-5mm}
\times \, \Big\{
e^{-\frac{\pi i(n_1+k_1)^2}{8(K+3)}} \, 
e^{-\frac{\pi i(n_2+k_2)^2}{8(K+3)}}
\, - \, 
e^{-\frac{\pi i(n_1+k_1)^2}{8(K+3)}} \, 
e^{-\frac{\pi i(n_2-k_2)^2}{8(K+3)}}
\\[2mm]
& &
\, + \, 
e^{-\frac{\pi i(n_1-k_1)^2}{8(K+3)}} \, 
e^{-\frac{\pi i(n_2+k_2)^2}{8(K+3)}}
\, - \, 
e^{-\frac{\pi i(n_1-k_1)^2}{8(K+3)}} \, 
e^{-\frac{\pi i(n_2-k_2)^2}{8(K+3)}} \Big\} 
\, \times \, 
\overset{w}{R}{}^{(+)} \cdot \, 
\overset{w}{\rm ch}{}^{(+)}_{H(\Lambda_{k_1,k_2}^{[K]})}
\\[2mm]
& & \hspace{-8mm}
= \,\
\Big[\frac{-i\gamma_{2(K+3)}}{4(K+3)}\Big]^2
\sum_{\substack{0 \, < \, k_1, \, k_2 \, < \, 2(K+3) \\[1mm]
k_1 \, \in \, \zzz_{\rm odd} \\[1mm]
k_2 \, \in \, \zzz_{\rm even} }} \hspace{-5mm}
(-1)^{\frac12(n_1+n_2+k_1+k_2)}
\\[2mm]
& & \hspace{-5mm}
\times \, \Big(
e^{-\frac{\pi i(n_1+k_1)^2}{8(K+3)}} \, 
+
e^{-\frac{\pi i(n_1-k_1)^2}{8(K+3)}}\Big) \, \Big(
e^{-\frac{\pi i(n_2+k_2)^2}{8(K+3)}} \, 
-
e^{-\frac{\pi i(n_2-k_2)^2}{8(K+3)}}\Big) \, 
\overset{w}{R}{}^{(+)} \cdot \, 
\overset{w}{\rm ch}{}^{(+)}_{H(\Lambda_{k_1,k_2}^{[K]})}
\end{eqnarray*}}
proving (i). \, The formula in (ii) is obtained by similar calculation.

\medskip

\noindent
2) (i) \,\ By Proposition \ref{paper:prop:2025-1105b} and Lemma 
\ref{lemma:2025-721b}, we have 
{\allowdisplaybreaks
\begin{eqnarray*}
& & \hspace{-5mm}
4 \, \Big(\overset{w}{R}{}^{(+)} \cdot \, 
\overset{w}{\rm ch}{}^{(+)}_{H(\Lambda_{n_1,n_2}^{[K]})}
\Big)\Big|_{T}
\\[2mm]
& & \hspace{-7mm}
= \Big\{
\big(f_{-n_1+(K+3),2(K+3)}^{(+)}-f_{n_1-(K+3),2(K+3)}^{(+)}\big)
+ 
\big(f_{-n_1+(K+3),2(K+3)}^{(-)}+f_{n_1-(K+3),2(K+3)}^{(-)}\big)\Big\}
\Big|_T
\\[2mm]
& & \hspace{-5mm}
\times \Big\{ 
\big(f_{-n_2+(K+3),2(K+3)}^{(+)}-f_{n_2-(K+3),2(K+3)}^{(+)}\big)
- 
\big(f_{-n_2+(K+3),2(K+3)}^{(-)}+f_{n_2-(K+3),2(K+3)}^{(-)}\big)\Big\}
\Big|_T
\\[3mm]
& & \hspace{-7mm}
= \,\ 
e^{\frac{\pi i}{4(K+3)}\{[n_1-(K+3)]^2+[n_2-(K+3)]^2\} }
\\[2mm]
& & \hspace{-5mm}
\times \Big\{
\big(f_{-n_1+(K+3),2(K+3)}^{(-)}-f_{n_1-(K+3),2(K+3)}^{(-)}\big)
+ 
\big(f_{-n_1+(K+3),2(K+3)}^{(+)}+f_{n_1-(K+3),2(K+3)}^{(+)}\big)\Big\}
\\[2mm]
& & \hspace{-5mm}
\times \Big\{ 
\big(f_{-n_2+(K+3),2(K+3)}^{(+)}-f_{n_2-(K+3),2(K+3)}^{(+)}\big)
- 
\big(f_{-n_2+(K+3),2(K+3)}^{(-)}+f_{n_2-(K+3),2(K+3)}^{(-)}\big)\Big\}
\\[3mm]
& & \hspace{-7mm}
= \,\ 
e^{\frac{\pi i}{4(K+3)}\{[n_1-(K+3)]^2+[n_2-(K+3)]^2\} }
e^{-\frac{\pi i}{2}(n_1-n_2)}
\, \times \, 
4 \, \overset{w}{R}{}^{(-)} \cdot \, 
\overset{w}{\rm ch}{}^{(-)}_{H(\Lambda_{n_1,n_2}^{[K]})}
\\[0mm]
& & \hspace{-7mm}
= \,\ - \, 
(-1)^{\frac{K+3}{2}} \, 
e^{\frac{\pi i}{4(K+3)}(n_1^2+n_2^2)}
\, \times \, 
4 \, \overset{w}{R}{}^{(-)} \cdot \, 
\overset{w}{\rm ch}{}^{(-)}_{H(\Lambda_{n_1,n_2}^{[K]})} 
\end{eqnarray*}}
proving (i). \, The formula in (ii) is obtained by similar calculation.
\end{proof}

\medskip

As to the denominators $\overset{w}{R}{}^{(\pm)}$ of the QHR 
in the case $C_2^{(1)}$ and $f=e_{-\theta}$, they are very easily
obtained from \eqref{paper:eqn:2025-1105a1} and 
\eqref{paper:eqn:2025-1105a2} and Lemma 
\ref{paper:lemma:2025-1105a}, since $\ell=2$, 
$\dim \overline{\ggg}_0=4$, $\dim \overline{\ggg}_{\frac12}=2$, 
$\dim \overline{\ggg}^f=6$, 
and $\overline{\Delta}_0^+ = \{\alpha_2 \}$ and 
$\overline{\Delta}_{\frac12} = \{\alpha_1, \, \alpha_1+\alpha_2 \}$.

\medskip

\begin{prop} 
\label{prop:2025-727c}
In the case $C_2^{(1)}$ and $f=e_{-\theta}$, the denominators of 
quantum Hamiltonian reduction and their modular transformation 
are given by the following formulas:
\begin{enumerate}
\item[{\rm 1)}]
\begin{enumerate}
\item[{\rm (i)}] \,\ $\overset{w}{R}{}^{(+)}(\tau,z,t)
\, = \,\ 
\widetilde{\vartheta}_{11}(\tau, 2z,t) \, 
\widetilde{\vartheta}_{01}(\tau,z,t)
\,\ = \,\ 
e^{4\pi it} \, 
\vartheta_{11}(\tau, 2z) \, 
\vartheta_{01}(\tau,z)$

\item[{\rm (ii)}] \,\ $\overset{w}{R}{}^{(-)}(\tau,z,t)
\, = \,\ 
\widetilde{\vartheta}_{11}(\tau, 2z,t) \, 
\widetilde{\vartheta}_{00}(\tau,z,t)
\,\ = \,\ 
e^{4\pi it} \, 
\vartheta_{11}(\tau, 2z) \, 
\vartheta_{00}(\tau,z)$
\end{enumerate}

\item[{\rm 2)}] 
\begin{enumerate}
\item[{\rm (i)}] \,\ $\overset{w}{R}{}^{(+)}|_{ST^2S}(\tau,z,t)
\, = \,\ 
- \, \overset{w}{R}{}^{(+)}(\tau,z,t)$ 

\item[{\rm (ii)}] \,\ $\overset{w}{R}{}^{(-)}|_{ST^2S}(\tau,z,t)
\, = \,\ 
i \, \overset{w}{R}{}^{(-)}(\tau,z,t)$ 

\item[{\rm (iii)}] \,\ $\overset{w}{R}{}^{(\pm)}|_{T}(\tau,z,t)
\, = \,\ 
e^{\frac{\pi i}{4}} \, \overset{w}{R}{}^{(\mp)}(\tau,z,t)$ 
\end{enumerate}
\end{enumerate}
\end{prop}

\medskip

By Propositions \ref{paper:prop:2025-1105b} and 
\ref{prop:2025-725a} and \ref{prop:2025-727c}, the characters 
of quantum Hamiltonian reduction and their modular transformation 
are obtained as follows:

\medskip

\begin{prop} \,\ 
\label{paper:prop:2025-1105c}
Under the same assumption with Proposition \ref{prop:2025-725a},
the following formulas hold:
\begin{enumerate}
\item[{\rm 1)}]
\begin{enumerate}
\item[{\rm (i)}] \,\ $
\overset{w}{\rm ch}{}^{(+)}_{H(\Lambda_{n_1,n_2}^{[K]})}
(\tau,z\alpha_2,t)$
{\allowdisplaybreaks
\begin{eqnarray*}
& & \hspace{-20mm}
= \,\ \Big\{
\big(f_{-n_1+(K+3),2(K+3)}^{(-)}+f_{n_1-(K+3),2(K+3)}^{(-)}\big)
+ 
\big(f_{-n_1+(K+3),2(K+3)}^{(+)}-f_{n_1-(K+3),2(K+3)}^{(+)}\big)\Big\}
\\[2mm]
& & \hspace{-18mm}
\times \, \Big\{
- \big(f_{-n_2+(K+3),2(K+3)}^{(-)}+f_{n_2-(K+3),2(K+3)}^{(-)}\big)
+ 
\big(f_{-n_2+(K+3),2(K+3)}^{(+)}-f_{n_2-(K+3),2(K+3)}^{(+)}\big)\Big\}
\\[2mm]
& & \hspace{-18mm}
/ \,\ 4 \, \widetilde{\vartheta}_{11}(\tau, 2z,t) \, 
\widetilde{\vartheta}_{01}(\tau, z,t) 
\end{eqnarray*}}

\item[{\rm (ii)}] \,\ $
\overset{w}{\rm ch}{}^{(-)}_{H(\Lambda_{n_1,n_2}^{[K]})}
(\tau,z\alpha_2,t) 
\,\ = \,\ e^{\frac{\pi i}{2}(n_1-n_2)}$
{\allowdisplaybreaks
\begin{eqnarray*}
& & \hspace{-20mm}
\times \, \Big\{
\big(f_{-n_1+(K+3),2(K+3)}^{(-)}-f_{n_1-(K+3),2(K+3)}^{(-)}\big)
+ 
\big(f_{-n_1+(K+3),2(K+3)}^{(+)}+f_{n_1-(K+3),2(K+3)}^{(+)}\big)\Big\}
\nonumber
\\[2mm]
& & \hspace{-20mm}
\times \, \Big\{
- 
\big(f_{-n_2+(K+3),2(K+3)}^{(-)}
+ 
f_{n_2-(K+3),2(K+3)}^{(-)}\big)
+ 
\big(f_{-n_2+(K+3),2(K+3)}^{(+)}
- 
f_{n_2-(K+3),2(K+3)}^{(+)}\big)\Big\}
\\[2mm]
& & \hspace{-18mm}
/ \,\ 4 \, \widetilde{\vartheta}_{11}(\tau, 2z,t) \, 
\widetilde{\vartheta}_{00}(\tau, z,t) 
\end{eqnarray*}}
\end{enumerate}

\vspace{1mm}
\item[{\rm 2)}]
\begin{enumerate}
\item[{\rm (i)}] \,\ $
\overset{w}{\rm ch}{}^{(+)}_{H(\Lambda_{n_1,n_2}^{[K]})}
\Big|_{ST^2S}
= \,\
\Big[\dfrac{\gamma_{2(K+3)}}{4(K+3)}\Big]^2
\sum\limits_{\substack{0 \, < \, k_1, \, k_2 \, < \, 2(K+3) \\[1mm]
k_1 \, \in \, \zzz_{\rm odd} \\[1mm]
k_2 \, \in \, \zzz_{\rm even} }} \hspace{-5mm}
(-1)^{\frac12(n_1+n_2+k_1+k_2)} $
$$
\times \, \Big(
e^{-\frac{\pi i(n_1+k_1)^2}{8(K+3)}} \, 
+
e^{-\frac{\pi i(n_1-k_1)^2}{8(K+3)}}\Big) \, \Big(
e^{-\frac{\pi i(n_2+k_2)^2}{8(K+3)}} \, 
-
e^{-\frac{\pi i(n_2-k_2)^2}{8(K+3)}}\Big) \, 
\overset{w}{\rm ch}{}^{(+)}_{H(\Lambda_{k_1,k_2}^{[K]})}
$$

\item[{\rm (ii)}] \,\ $
\overset{w}{\rm ch}{}^{(-)}_{H(\Lambda_{n_1,n_2}^{[K]})}
\Big|_{ST^2S}
= \,\ - \, i \, 
\Big[\dfrac{\gamma_{2(K+3)}}{4(K+3)}\Big]^2
\sum\limits_{\substack{0 \, < \, k_1, \, k_2 \, < \, 2(K+3) \\[1mm]
k_1 \, \in \, \zzz_{\rm odd} \\[1mm]
k_2 \, \in \, \zzz_{\rm even} }} $
$$
\times \, \Big(
e^{-\frac{\pi i(n_1+k_1)^2}{8(K+3)}} \, 
+
e^{-\frac{\pi i(n_1-k_1)^2}{8(K+3)}}\Big) \, \Big(
e^{-\frac{\pi i(n_2+k_2)^2}{8(K+3)}} \, 
-
e^{-\frac{\pi i(n_2-k_2)^2}{8(K+3)}}\Big) \, 
\overset{w}{\rm ch}{}^{(-)}_{H(\Lambda_{k_1,k_2}^{[K]})}
$$
\end{enumerate}
\item[{\rm 3)}]
\begin{enumerate}
\item[{\rm (i)}] \,\ $
\overset{w}{\rm ch}{}^{(+)}_{H(\Lambda_{n_1,n_2}^{[K]})}
\Big|_{T}
= \,\
- \, e^{-\frac{\pi i}{4}} \, (-1)^{\frac{K+3}{2}} \, 
e^{\frac{\pi i(n_1^2+n_2^2)}{4(K+3)}} \,\ 
\overset{w}{\rm ch}{}^{(-)}_{H(\Lambda_{n_1,n_2}^{[K]})}$

\vspace{1.5mm}

\item[{\rm (ii)}] \,\ $
\overset{w}{\rm ch}{}^{(-)}_{H(\Lambda_{n_1,n_2}^{[K]})}
\Big|_{T}
= \,\ e^{-\frac{\pi i}{4}} \, 
(-1)^{\frac{K+3}{2}} \, 
e^{\frac{\pi i(n_1^2+n_2^2)}{4(K+3)}} \,\ 
\overset{w}{\rm ch}{}^{(+)}_{H(\Lambda_{n_1,n_2}^{[K]})}$
\end{enumerate}
\end{enumerate}
\end{prop}

\end{document}